\documentclass[11pt]{amsart}

\usepackage{mathrsfs} 
\usepackage{mathtools}
\usepackage{a4wide} 
\usepackage{graphicx} 
\usepackage{amssymb} 
\usepackage{epstopdf}
\usepackage{enumerate} 
\usepackage{titletoc} 
\usepackage[colorlinks=true, urlcolor=blue, linkcolor=blue, citecolor=black]{hyperref} %
\usepackage[nameinlink]{cleveref} 
\usepackage{esint} 
\usepackage{verbatim}
\usepackage{enumitem}
\usepackage{kotex}
\usepackage{amsfonts}
\usepackage{amsmath}
\usepackage{amssymb}
\usepackage{amsthm}
\usepackage{enumerate}
\usepackage{esint}
\usepackage{fancyhdr}
\usepackage{setspace}
\usepackage{bbm}
\usepackage[title]{appendix}
\usepackage{mathtools}
\usepackage{graphics,amscd,latexsym,empheq,mathtools,caption}
\usepackage{tikz}

\usetikzlibrary{arrows.meta}

\crefformat{equation}{#2(#1)#3}
\Crefformat{equation}{\textup{#2(#1)#3}}

\numberwithin{equation}{section}

\theoremstyle{plain}
\newtheorem{theorem}{Theorem}[section]
\newtheorem{lemma}[theorem]{Lemma}

\newtheorem{proposition}[theorem]{Proposition}

\theoremstyle{definition}
\newtheorem{definition}[theorem]{Definition}

\makeatletter
\@namedef{subjclassname@2020}{\textup{2020} Mathematics Subject Classification}
\makeatother

\title[Diffusion-Reaction Epidemic Model with a Free Boundary]{Diffusion-Reaction Epidemic Model with a Free Boundary}

\author{Aesol Jeon}
\address{The National Institute for Mathematical Sciences, Daejeon 34047, Korea}
\email{aesol@nims.re.kr}

\author{Ki-Ahm Lee}
\address{Department of Mathematical Sciences, Seoul National University, Seoul 08826, Korea}
\email{ki-ahm@snu.ac.kr}
  

\thanks{This work was supported by the Ministry of Education of the Republic of Korea and the National Research Foundation of Korea (RS-2025-00515707)}
\begin{document}
\begin{abstract}
	This study investigates an SEIS PDE model with a free boundary, which captures the dynamics of epidemic transmission, including diseases like COVID-19. This parabolic PDE system is analyzed in a rotationally symmetric domain, and the existence and uniqueness of the local solution are established through the straightening lemma. Furthermore, the existence and uniqueness of the global solution are established under specific conditions on the diffusion coefficients. Then the model introduces the basic reproductive number, $R_0$, which provides sufficient conditions for determining whether the disease will vanish or spread. Notably, when $R_0<1$, the disease-free equilibrium(DFE) is shown to be globally stable, and when $R_0>1$, the DFE is unstable. Lastly, we investigate the convergence speed of solutions by applying nonlinear elliptic eigenvalue techniques to the associated parabolic PDE system.
\end{abstract}

\maketitle

%==========================================================================
% 
%							Introduction
% 
%==========================================================================
\section{Introduction}
Nowadays, many infectious diseases occur, such as SARS and malaria. Thus, mathematical modeling of epidemics and their quantitative and qualitative prediction have received much attention in various areas, especially in mathematical biology. 
Furthermore, about four years ago, there was a great pandemic that threatened the whole world, from Asia to America. The name of the disease is COVID-19, and it affected our whole lives, for example, almost all offices and factories were shut down, so it became quite a significant loss of human and material resources. We had a great interest in predicting this disease and thought that mathematical modeling would make our lives easier. In contrast to earlier epidemics, the COVID-19 period provided abundant data suitable for modeling and prediction, which encouraged extensive mathematical and theoretical research by scholars from diverse disciplines. Moreover, we expect that it can evaluate the impact of interventions such as vaccination or social distancing.

A variety of models, commonly referred to as compartmental systems in epidemiology, have been proposed and applied \cite{Kermack_1991}. Each model has its own characteristics, as outlined in the following table. Before proceeding with research about the COVID-19 model, we briefly summarize what work has been done so far.
By going through this process, we conclude that SEIS is the most appropriate model for the COVID-19 virus, since the virus has no immunity after recovery and a latent period before infection.
	\begin{table}[h] 
	\centering
	\captionsetup{position=below}
	\caption{Various epidemic models}
		\begin{tabular}{ccc}
			\hline
Model name & Special Property                                                                                  & Example of disease \\ \hline
SI         & No recovery                                                                                       & AIDS               \\
SIR        & Permanent immunity after recovery                                                                 & Influenza          \\
SIS        & No immunity after recovery                                                                        & H1N5               \\
SEIR       & \begin{tabular}[c]{@{}c@{}}Latent period where the infected\\  is not infectious yet\end{tabular} & Measles            \\
SEIS       & SIS+latent period                                                                                 & Malaria            \\ \hline
		\end{tabular}
\end{table}
			
Until now, most of the epidemic models--the well-known compartmental epidemic model--have assumed that the spread of infectious diseases depends only on time, which leads to a simple system of ODEs. However, in the real world, disease transmission is also affected by the spatial distribution of pathogens and hosts. Thus, we constructed an epidemic model considering both temporal and spatial differences. Although our ultimate goal of this research is to deal with the model in general $n$-dimensional space, this first attempt only introduces the rotational symmetric case.
	
To construct a spatio-temporal SEIS model, we first examined the classical SEIS model, which accounts only for temporal dynamics. Spatial effects were incorporated by adding diffusion terms to each compartment. In addition, to control the spread rate of the infection, a free boundary was introduced to restrict the region in which infected and exposed individuals can exist. Without such a boundary, the disease would rapidly spread over the entire domain due to the properties of the heat equation. Subsequently, we analyzed the proposed model to establish the existence and uniqueness of solutions, which is a standard objective in PDE. To solve the diffusion-reaction model with a free boundary, we followed previous studies on the SIR diffusion-reaction model \cite{Kim_2013}, as the SIR framework remains one of the most fundamental and widely used epidemic models. Furthermore, many mathematical biologists have focused on the asymptotic behavior of such systems, since it provides valuable insights into disease persistence and extinction. Accordingly, we defined the reproductive number for our model, computed its value, and demonstrated that the asymptotic behavior depends on this parameter. 
	
\subsection{Organization of Paper}
The remainder of this paper is organized as follows. In section 2, we introduce the SEIS model and its generalization to a spatially dependent PDE, including the incorporation of a free boundary. The process of construction of a free boundary is explained in the context. Section 3 discusses the existence and uniqueness of solutions, demonstrating the global existence and uniqueness under specific conditions on the diffusion coefficients by using uniform estimates. In section 4, we analyze the asymptotic behavior of the model by defining the reproductive number as the bifurcation parameter for disease transmission and examining the equilibrium states based on this number. This section also contains the convergence speed of the convergent solution by using elliptic eigenvalue problems and energy minimizer.
	
\subsection{Notation}
Prior to the main results, we introduce some useful notation that will be used throughout the paper. It is well known that H\"older continuity of solutions means uniform continuity with respect to
\begin{equation*}
	\rho(X,Y)=\sqrt{\lvert t-s \rvert}+\lvert x-y \lvert
\end{equation*}
where $X=(t,x)$ and $Y=(s,y)$. It means that solutions are always twice more regular with respect to the space variable than with respect to the time variable.
Thus following \cite{Krylov_1996} and \cite{Boucksom_2013}, we will write $u\in C^{0,\alpha}$ for functions that are $\alpha$-H\"older continuous in $x$ and $\frac{\alpha}{2}$-H\"older continuous in $t$.

Usually, parabolic equations are considered in the cylindrical domain. Therefore let $D_T=D \times (0,T]$ for time $T \in (0,\infty]$ be a time-space domain and $\alpha \in (0,1]$. Then $u \in C^{\alpha}(\overline{D_T})$ means that there exists $C>0$ such that for all $(t,x), (s,y) \in D$,
\begin{equation*}
	\lvert u(t,x)-u(s,y)\rvert \leq C(\lvert t-s \rvert ^\frac{\alpha}{2}+\lvert x-y \rvert ^\alpha).
\end{equation*}

And $u \in C^{1+\alpha}(\overline{D_T})$ if $u$ is $\frac{\alpha +1}{2}$-H\"older continuous only in $t$ and $Du$ is $\alpha$-H\"older continuous. Similarly, $u \in C^{2+\alpha}(\overline{D_T})$ if $\frac{\partial u}{\partial t}$ and $D^2u$ are $\alpha$-H\"older continuous.
Likewise, define H\"older semi-norms,
\begin{equation*}
\begin{aligned}
	&[u]_{C^{\alpha}(\overline{D_T})}=\sup_{X,Y\in D_T,\ X \neq Y} \frac{\lvert u(X)-u(Y) \rvert}{{\rho (X,Y)}^\alpha},
	&&[u]_{C_t^{\alpha}(\overline{D_T})}=\sup_{\substack{t,s \in (0,T),\ t\neq s,\\ x\in D}}\frac{\lvert u(t,x)-u(s,x) \rvert}{\lvert t-s \rvert ^\frac{\alpha}{2}},\\
	&[u]_{C^{1+\alpha}(\overline{D_T})}=[u]_{C^{1+\alpha}_t(\overline{D_T})}+[Du]_{C^\alpha(\overline{D_T})},
	&&[u]_{C^{2+\alpha}(\overline{D_T})}=\left[\frac{\partial u}{\partial t}\right]_{C^\alpha(\overline{D_T})}+[D^2u]_{C^\alpha(\overline{D_T})},
\end{aligned}
\end{equation*}
and H\"older norms,
\begin{equation*}
\begin{gathered}
	\lVert u \rVert_{C^0(\overline{D_T})}=\sup_{X \in D_T} \lvert u(X) \rvert, \\
	\lVert u \rVert_{C^{1+\alpha}(\overline{D_T})}=\lVert Du \rVert_{C^0(\overline{D_T})}+\lVert u\rVert_{C^0(\overline{D_T})}+[u]_{C^{1+\alpha}(\overline{D_T})}, \\
	\lVert u \rVert_{C^{2+\alpha}(\overline{D_T})}=\lVert D^2u \rVert_{C^0(\overline{D_T})}+\lVert Du\rVert_{C^0(\overline{D_T})}+\left \lVert \frac{\partial u}{\partial t} \right \rVert_{C^0(\overline{D_T})}+\lVert u\rVert_{C^0(\overline{D_T})}+[u]_{C^{2+\alpha}(\overline{D_T})}.
\end{gathered}
\end{equation*}

However, since we will usually deal with domains either inside the free boundary or in the whole space, we use
\begin{equation*}
\begin{alignedat}{6}
&\overline{\Omega_{hT}} &&= [0,h(t)]\times[0,T],\qquad
&&\overline{\Omega_{\infty T}} &&= [0,\infty)\times[0,T],\qquad
&&\overline{\Omega_{hT}} &&= [0,h(t))\times(0,T],\\[6pt]
&\overline{\Omega_{\infty T}} &&= [0,\infty)\times(0,T],\qquad
&&\Omega_{hT} &&= (0,h(t))\times(0,T],\qquad
&&\Omega_{\infty T} &&= (0,\infty)\times(0,T],\\[6pt]
&\Omega_h &&= (0,h(t))\times(0,\infty),\qquad
&&\Omega_{\infty} &&= (0,\infty)\times(0,\infty),\qquad
&&\Gamma_h &&= \{h(t)\}\times(0,\infty),\\[6pt]
&\Lambda_h &&= (0,h(t))\times[T_0,\infty),\qquad
&&\overline{\Lambda_\infty} &&= [0,\infty)\times[T_0,\infty).
\end{alignedat}
\end{equation*}
At this point, since the boundary data plays an important role on each parabolic domain, one must carefully examine the statements of theorems with respect to each domain.
Since the model system of equations to be addressed later is a parabolic-type PDE, we will use the Hölder norms defined earlier.

\section{SEIS Diffusion-Reaction Model with Free Boundary}
\

It is a well-known fact that the classical SEIS epidemic model is constructed by 
	\begin{equation} \label{SEIS ODE}
		\left \{ \begin{aligned}
			S'(t)&=A-\alpha I(t)S(t)-\mu_1 S(t)+r_2 I(t)+r_1 E(t),
			\\
			E'(t)&=(1-p)\alpha I(t)S(t)-\beta_1 E(t)-r_1 E(t) -\mu_2 E(t),
			\\
			I'(t)&=p \alpha I(t)S(t)+\beta_1 E(t)-r_2 I(t)-\mu_3 I(t),
			\\
			S(0)&=N, \ E(0)\geq 0, \ I(0)\geq 0
		\end{aligned} \right.
	\end{equation}
	in the biological meaning region $D=\{(S,E,I) \in \mathbb{R}^3 :S>0,E \geq 0,I \geq 0\}$ \vspace{0.5cm} where the coefficient of the equation is that\\	 
	\
	
	\begin{center} \begin{tabular}{c|l}
	$\mu_1,\mu_2, \mu_3$      & death rate for physical disease of the exposed and the infected                                       \\ \hline
	$N$          & original population size                                              \\ \hline
	$A$          & influx or recruitment of the susceptible and the exposed              \\ \hline
	$r_1, r_2$ & treatment cure rate of latent and active disease                      \\ \hline
	$\beta_1$    & breakdown rate from latent to an active condition                        \\ \hline
	$\alpha$   & incidence of meeting                                                  \\ \hline
	$p$          & a portion of infection instantaneously degeneration in an active condition
	\end{tabular}\end{center}
	followed by \cite{MENG_2013}.
	
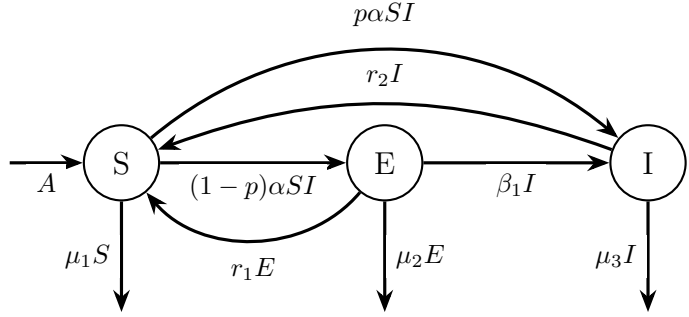
	\begin{center} \begin{tikzpicture}
            \begin{scope}[every node/.style={circle,thick,draw,minimum size=10mm,font=\large}]
                \node (A) at (-4,-1) [draw=none] { };
                \node (S) at (-2,-1) {S};
                \node (E) at (1.5,-1) {E};
                \node (I) at (5,-1) {I};
                \node (S2) at (-2,-3.5) [draw=none] { };
                \node (E2) at (1.5,-3.5) [draw=none] { };
                \node (I2) at (5,-3.5) [draw=none] { };
            \end{scope}

            \begin{scope}[>={Stealth[black]},
                          every node/.style={fill=white,rectangle,minimum size=0.5mm,font=\small},
                          every edge/.style={draw=black, very thick}]
                \path [->] (A) edge node [below] {$A$} (S);
                \path [->] (S) edge node [below]{$(1-p)\alpha SI$} (E);
                \path [->] (E) edge node [below] {$\beta_1 I$} (I);
                \path [->] (S) edge node [left] {$\mu_1 S$} (S2);
                \path [->] (E) edge node [right]{$\mu_2 E$} (E2);
                \path [->] (I) edge node [left]{$\mu_3 I$} (I2);
            \end{scope}
            \begin{scope}[>={Stealth[black]},
                          every node/.style={circle,minimum size=0.5mm,font=\small},
                          every edge/.style={draw=black, very thick}]
                          \path [->] (I) edge [bend right=20] node [above,yshift=-1mm] {$r_2I$} (S);
                \path [->] (S) edge [out=40,in=140] node [above,yshift=-1.5mm]  {$p\alpha SI$} (I);
                \path [->] (E) edge [bend left=50] node [below,yshift=1.5mm]{$r_1E$} (S);
                \end{scope}

        \end{tikzpicture} \end{center}
	
	If we assume that $\mu_1=\mu_2=\mu_3=\mu$ for some constant $\mu>0$, i.e., death rates of individuals of other groups are equal, we can get
	\begin{equation*}
		(S+E+I)'=\mu\left(\frac{A}{\mu}-S-E-I\right),
	\end{equation*}
	so it follows that
	\begin{equation*}
		\lim_{t\to \infty}(S+E+I)=\frac{A}{\mu}.
	\end{equation*}
	
	Then the system of equations \cref{SEIS ODE} is restricted in $\Omega \in \mathbb{R}_+^3$ where
	\begin{equation} \label{omega}
		\Omega = \left\{ (S,E,I)\in \mathbb{R}_+^3 : S>0, E\geq 0, I \geq 0, S+E+I\leq \frac{A}{\mu}\right\}.
	\end{equation}
	
	We want to generalize our ODE to PDE with diffusion terms since we want to consider an additional spatial change. So we can construct the equations

	\begin{equation} \label{diffusion SEIS}
		\left \{ \begin{aligned}	
		&S_t-d_1\Delta S=A-\alpha IS-\mu_1 S+r_2 I+r_1E,&&x \in \Omega,\\
		&E_t-d_2\Delta E=(1-p)\alpha IS-\beta_1 E-r_1E-\mu_2 E,&&x \in \Omega,\\
		&I_t-d_3\Delta I=p\alpha IS+\beta_1 E-r_2I-\mu_3 I, &&x\in \Omega,\\
		\end{aligned} \right.
 	\end{equation}
 	for $t>0$, with initial conditions
 	\begin{equation} \label{cond1}
 		S(x,0)=S_0(x),  \quad E(x,0)=E_0(x), \quad I(x,0)=I_0(x), \quad x \in \overline \Omega
 	\end{equation}
 	and boundary conditions
 	\begin{equation} \label{cond2}
 		\partial_nS=\partial_nE=\partial_nI=0,\quad x\in \partial \Omega,\ t>0,
 	\end{equation}
 	where $\Omega$ is a fixed and bounded domain in $\mathcal{R}^n$ with smooth boundary $\partial \Omega$.
 	
 	Then we can guess that the solution $I(r,t)$ of system \cref{diffusion SEIS} with \cref{cond1} and \cref{cond2} are always positive for any time $t>0$ whatever the nonnegative nontrivial initial data is. It can be interpreted as the disease spreading to the whole area immediately, even when the infection is confined to a small part of the area in the beginning. Then, it is not the observed fact that disease always spreads gradually. So, we consider the epidemic model with a free boundary, which describes the spreading frontier of the disease.
 	 
 	For these reasons, we will observe the properties of frontier to construct a proper free boundary of the model. Thus when we imagine a spread of disease, infection occurs equally regardless of directions, for simplicity, we assume the environment is radially symmetric. Since the diffusivity of the exposed and the infected are $d_2$ and $d_3$ respectively, and the number of the flowing exposed and the flowing infected on any point $x\in \Omega$ are $\partial E/ \partial x$ and $\partial I/ \partial x$ respectively, the number of the infected flowing across the boundary $x=h(t)$ at time $t$ to $t+\Delta t$ is 
 	\begin{equation} \label{time difference free boundary}
 		-\left(d_2 \frac{\partial E}{\partial x}+d_3 \frac{\partial I}{\partial x}\right) \Delta t.
 	\end{equation}
 	The exposed and the infected disperse from $x=h(t)$ to $x=h(t+\Delta t)$ during the time interval $[t,t+\Delta t]$ and the size of the population decides the distance $h(t+\Delta t)-h(t)$. Thus we can suppose that \cref{time difference free boundary} equals to
 	\begin{equation*}
 		f(h(t+\Delta t)-h(t)).
 	\end{equation*}
 	Then we know that the function $f$ is increasing and $f(0)=0$. Biologically, this means that the size is increasing with respect to the moving length. Using the Taylor expansion of the function $f$ yields that
 	\begin{equation*}
 		f(h(t+\Delta t)-h(t))=f'(0)(h(t+\Delta t)-h(t))+\textstyle\frac{1}{2}f''(0)(h(t+\Delta t)-h(t))^2+\cdots,
 	\end{equation*}
 	so
 	\begin{equation*}
 		-\left(d_1 \frac{\partial E}{\partial x}+d_2 \frac{\partial I}{\partial x}\right)=f'(0) \frac{h(t+\Delta t)-h(t)}{\Delta t}+\frac{1}{2}f''(0)\frac{(h(t+\Delta t)-h(t))^2}{\Delta t}+ \cdots.
 	\end{equation*}
 	If $\Delta t \to 0$, we can deduce the result that
 	\begin{equation*}
 		-\left(d_2 \frac{\partial E}{\partial x}+d_3 \frac{\partial I}{\partial x}\right)=f'(0)h'(t).
 	\end{equation*}
 	
 	Now, $f'(0)>0$ by the definition of $f$ and depends on the diffusivity of the exposed $d_2$ and of the infected $d_3$. Furthermore, we can know that the bigger $f'(0)$, the easier the infected disperse to the unpolluted area.
 	
 	Let $\mu = \frac{d_3}{f'(0)}$ and $\beta = \frac{d_2}{f'(0)}$. Then we can induce the conditions on the free boundary
 	\begin{equation*}
 		-\left(\beta\frac{\partial E}{\partial x}+\mu \frac{\partial I}{\partial x}\right)= h'(t) \quad \text{and} \quad E=I=0.
 	\end{equation*}
 	
 	Supplementally, if all populations do not emigrate from inside, then there is no flux crossing the boundary, so the Neumann boundary conditions arise to describe these phenomena
 	\begin{equation*}
 		\frac{\partial S}{\partial r}=\frac{\partial E}{\partial r}=\frac{\partial I}{\partial r}=0 \quad \text{on } \partial \Omega.
 	\end{equation*}
 	
 	Then we will investigate the behavior of the positive solution ($S\left(r,t\right)$, $I\left(r,t\right)$, $R\left(r,t\right)$; $h\left(t\right)$) with $r=\lvert x \rvert$ and $x \in \mathbb{R}^n$ in the following problem :
 	\begin{equation} \label{free boundary SEIS}
 	\begin{cases}
			S_t-d_1\Delta S=A-\alpha IS-\mu_1 S+r_2 I+r_1E, &(r,t)\in\Omega_\infty,\\
			E_t-d_2\Delta E=(1-p)\alpha IS-(\beta_1 +r_1+\mu_2) E, &(r,t) \in \Omega_h,\\
			I_t-d_3\Delta I=p\alpha IS+\beta_1 E-(r_2+\mu_3) I, &(r,t) \in \Omega_h,\\
			h'(t)=-\beta E_r(h(t),t)-\mu I_r(h(t),t), &t>0\\
			E(r,t)=I(r,t)=0, &(r,t)\notin \Omega_h,\\
			S_r(0,t)=E_r(0,t)=I_r(0,t)=0 &t>0,
	\end{cases}
	\end{equation}
	without the last equation, with initial conditions
	\begin{equation} \label{initial cond}
		\left \{ \begin{aligned}
			&h(0)=h_0>0, \\
			&S(r,0)=S_0,\ E(r,0)=E_0,\ I(r,0)=I_0, \qquad r\geq0, \\
		\end{aligned} \right.
	\end{equation}
 	where $\Delta_r w = w_{rr}+\frac{n-1}{r}w_r$. It can be easily obtained by using spherical coordinate change and rotational symmetry of $x$. Moreover, the initial functions $S_0, E_0,$ and $I_0$ are assumed to be nonnegative and satisfy
 	\begin{equation} \label{cond for initial cond}
 		\left \{ \begin{aligned}
 			&S_0 \in C^2([0,+\infty)),\\
 			&E_0, I_0 \in C^2([0,h_0)),\\
 			&E_0(r)=I_0(r)=0, &&r \in[h_0,+\infty),\\
 			&I_0(r)>0, && r \in [0,h_0).
 		\end{aligned} \right.
 	\end{equation}

\section{Well--posedness of the Model of Rotationally Symmetric Case}

When developing a mathematical model, researchers--particularly in the study of partial differential equations--first aim to establish the existence of solutions, even without knowing their explicit form. Subsequent analysis usually addresses properties such as uniqueness and regularity, which indicate the model's well-posedness and mathematical significance. In line with this approach, we first prove the short-time existence of a solution to our model, that is, we show that there exists a time $T_0$ such that the solution exists on $[0,T_0]$. We then investigate the uniqueness and regularity of the solution and the free boundary, and later aim to extend the existence time by deriving uniform estimates. The reinfection phenomenon complicates the derivation of such estimates; hence, we impose additional conditions on the diffusion coefficients to address this issue.

\subsection{Local Well--posedness}

	\begin{theorem}\label{shorttime}
		For any given $(S_0,E_0,I_0)$ satisfying \cref{cond for initial cond} and any $\gamma \in (0,1)$, there is a $T>0$ such that problem \cref{free boundary SEIS} with \cref{initial cond} admits a unique bounded solution \begin{equation*}
			(S, E, I; h) \in C^{1+\gamma}(\overline\Omega_T) \times [C^{1+\gamma}(\overline\Omega_{hT})]^2 \times C^{1+\gamma/2}([0,T]).
			\end{equation*}
			Moreover,
			\begin{equation*}
			\lVert S \rVert_{C^{1+\gamma}(\overline\Omega_T)}+\lVert E \rVert _{C^{1+\gamma}(\overline\Omega_{hT})} + \lVert I \rVert _{C^{1+\gamma}(\Omega_{hT})}+\Vert h \rVert _{C^{1+\gamma/2}([0,T])} \leq C.
	\end{equation*}
	Here $C$ and $T$ only depend on $h_0, \gamma, \lVert S_0 \rVert , \lVert E_0 \rVert $ and $\lVert I_0 \rVert$.
	\end{theorem}
To prove the theorem, we have to address the free boundary as a function of time $t$. To simplify this process, we first straighten the free boundary, a common technique for handling free boundaries. By applying certain diffeomorphisms, the free boundary can be transformed into a fixed boundary in time $t$.

\begin{lemma} \label{straightening lem}
	A free boundary $h(t)$ of the system of equations \eqref{free boundary SEIS} with \cref{initial cond} is diffeomorphic to $h_0$.
\end{lemma}

\begin{proof}
	We first straighten the free boundary. Construct $\xi(s)$ be a function in $C^3([0, \infty))$ satisfying
		\begin{equation*}
		\left \{ \begin{aligned}
			&\xi(s)=1, &&\text{\ if \ }\lvert s-h_0 \rvert < \frac{h_0}{8},
			\\
			&\xi(s)=0, &&\text{\ if \ }\lvert s-h_0 \rvert > \frac{h_0}{2},
			\\
			&\lvert \xi '(s) \rvert < \frac{5}{h_0}, &&\text{\ for all\ } s.
		\end{aligned} \right.
		\end{equation*}
		Consider the transformation
		\begin{equation*}
			(y,t) \mapsto (x,t), \text{ where } x=y+\xi(\lvert y \rvert )\left(h(t)-h_0 \frac{y}{\lvert y \rvert}\right), \ y \in \mathbb{R}^n,
		\end{equation*}
		which leads to the transformation
		\begin{equation*}
			(s,t) \mapsto (r,t)\text{ with } r=s+\xi(s)(h(t)-h_0),\ 0 \leq s < \infty.
		\end{equation*}
		As long as
		\begin{equation*}
			\lvert h(t)-h_0 \rvert \leq \frac{h_0}{8},
		\end{equation*}
		the above transformation $x \rightarrow y$ is a diffeomorphism from $\mathbb{R}^n$ onto $\mathbb{R}^n$ and the transformation $s \rightarrow r$ is also a diffeomorphism from $[0,+\infty)$ onto $[0,+\infty)$. So $s$ is radially symmetric like $r$ and we can define $\Delta_s w = w_{ss}+ \frac{n-1}{s}w_s$. Moreover, it changes the free boundary $r=h(t)$ to fixed boundary $s=h_0$.
\end{proof}

By applying \cref{straightening lem}, we transform the free boundary problem into one with a fixed boundary, where the transformed equation still retain information about the original free boundary. The proof of \cref{shorttime} is then obtained with the help of \cref{straightening lem}.
	\begin{proof}
		By using \cref{straightening lem}, free boundary $h(t)$ is changed to $h_0$. Now, set $u,v,$ and $w$ such that
		\begin{align*}
			S(r,t)=S(s+\xi(s)(h(t)-h_0),t) &\eqqcolon u(s,t),
			\\
			E(r,t)=E(s+\xi(s)(h(t)-h_0),t) &\eqqcolon v(s,t),
			\\
			I(r,t)=I(s+\xi(s)(h(t)-h_0),t) &\eqqcolon w(s,t).
		\end{align*}
		Then if we differentiate both sides with respect to $s$
		\begin{equation*}
				(1+\xi'(s)(h(t)-h_0))S_r=u_s,
		\end{equation*}
		so we can obtain
		\begin{equation*}
			S_r=\frac{u_s}{1+\xi'(s)(h(t)-h_0)}.
		\end{equation*}
		And if we differentiate both sides twice with respect to $s,$
		\begin{equation*}
			\begin{aligned}
			&\xi''(s)(h(t)-h_0)S_r+(1+\xi'(s)(h(t)-h_0))\frac{\partial}{\partial s}S_r=u_{ss}
			\\
			&\xi''(s)(h(t)-h_0)S_r+(1+\xi'(s)(h(t)-h_0))\frac{\partial r}{\partial s}\frac{\partial}{\partial r}S_r=u_{ss}
			\\
			&\xi''(s)(h(t)-h_0)S_r+(1+\xi'(s)(h(t)-h_0))^2S_{rr}=u_{ss}, 
			\end{aligned}
		\end{equation*}
		so we can obtain
		\begin{equation*}
			S_{rr}=\frac{u_{ss}(1+\xi'(s)(h(t)-h_0)-\xi''(s)(h(t)-h_0)u_s}{(1+\xi'(s)(h(t)-h_0))^3}.
		\end{equation*}
		Furthermore, if we differentiate both sides with respect to $t$,
		\begin{equation*}
			\begin{aligned}
				&\xi(s)h'(t)S_r+S_t=u_t
				\\
				&\frac{\xi(s)h'(t)u_s}{1+\xi'(s)(h(t)-h_0)}+S_t=u_t,
			\end{aligned}
		\end{equation*}
		so we can obtain
		\begin{equation*}
			S_t=u_t-\frac{\xi(s)h'(t)u_s}{1+\xi'(s)(h(t)-h_0)}.
		\end{equation*}
		By using the result above, we can change the equation of $S$, first line of \cref{free boundary SEIS}, to
		\begin{equation*}
			\begin{aligned}
				S_t-d_1\Delta_rS &=S_t-d_1\left(S_{rr}+\frac{n-1}{r}S_r\right)
				\\
				&=u_t-\frac{\xi(s)h'(t)u_s}{1+\xi'(s)(h(t)-h_0)}
				-d_1\frac{u_{ss}(1+\xi'(s)(h(t)-h_0)-\xi''(s)(h(t)-h_0)u_s}{(1+\xi'(s)(h(t)-h_0))^3}
				\\
				&\quad -d_1\frac{n-1}{r}\frac{u_s}{1+\xi'(s)(h(t)-h_0)}
				\\
				&=u_t-\frac{\xi(s)h'(t)u_s}{1+\xi'(s)(h(t)-h_0)}-\frac{d_1 u_{ss}}{(1+\xi'(s)(h(t)-h_0)^2}
				\\
				&\quad +\frac{d_1\xi''(s)(h(t)-h_0)u_s}{(1+\xi'(s)(h(t)-h_0)^3}-d_1 \frac{n-1}{r}\frac{u_s}{1+\xi'(s)(h(t)-h_0)}
				\\
				&=u_t-\frac{\xi(s)h'(t)u_s}{1+\xi'(s)(h(t)-h_0)}-\frac{d_1 \Delta_s u}{(1+\xi'(s)(h(t)-h_0)^2}
				+\frac{d_1\xi''(s)(h(t)-h_0)u_s}{(1+\xi'(s)(h(t)-h_0)^3}\\
				&\quad -d_1 \frac{n-1}{r}\frac{u_s}{1+\xi'(s)(h(t)-h_0)}
				+d_1\frac{n-1}{s}\frac{u_s}{(1+\xi'(s)(h(t)-h_0)^2}
				\\
				&=u_t-\frac{\xi(s)h'(t)u_s}{1+\xi'(s)(h(t)-h_0)}-\frac{d_1 \Delta_s u}{(1+\xi'(s)(h(t)-h_0)^2}
				+\frac{d_1\xi''(s)(h(t)-h_0)u_s}{(1+\xi'(s)(h(t)-h_0)^3}
				\\
				&\quad +d_1(n-1)\frac{-s\xi'(s)(h(t)-h_0)+\xi(s)(h(t)-h_0)}{(1+\xi'(s)(h(t)-h_0))^2s(s+\xi(s)(h(t)-h_0))}
				\\
				&=A-\alpha uw-\mu_1u+r_2w+r_1v.
			\end{aligned}
		\end{equation*}
		For simplicity, define some functions of the coefficients $X, Y, Z,$ and $W$ such that
		\begin{align*}
			\frac{\partial s}{\partial r}=\frac{1}{1+\xi'(s)(h(t)-h_0)}&\eqqcolon \sqrt{X(h(t),s)},
			\\
			\frac{\partial^2s}{\partial r^2}=-\frac{\xi''(s)(h(t)-h_0)}{[1+\xi'(s)(h(t)-h_0)]^3}&\eqqcolon Y(h(t),s),
			\\
			-\frac{1}{h'(t)}\frac{\partial s}{\partial t}=\frac{\xi(s)}{1+\xi'(s)(h(t)-h_0)} &\eqqcolon Z(h(t),s),
			\\
			(n-1)A\frac{(s\xi'-\xi)(h(t)-h_0)}{s[s+\xi(s)(h(t)-h_0)]} &\eqqcolon W(h(t),s).
		\end{align*}
		Then we can rewrite the equation above
		\begin{equation*}
			u_t-Xd_1\Delta_s u-(Yd_1+h'Z+Wd_1)u_s=A-\alpha uw-\mu_1u+r_2w+r_1v.
		\end{equation*}
		Similary, we had direct calculations for $E$ and $I$ respectively and it is changed to
		\begin{equation*}
			\begin{aligned}
				v_t-Xd_2\Delta_s v-(Yd_2+h'Z+Wd_2)v_s=(1-p)\alpha uw-\beta v-r_1v-\mu_2v
				\\
				w_t-Xd_3\Delta_s w-(Yd_3+h'Z+Wd_2)w_s=p\alpha uw+\beta v-r_2w-\mu_3w.
			\end{aligned}
		\end{equation*}
		Therefore the free boundary problem \cref{free boundary SEIS} with \cref{initial cond} becomes
		\begin{equation} \label{modified free boundary SEIS1}
			\begin{cases}
				u_t-Xd_1\Delta_s u-(Yd_1+h'Z+Wd_1)u_s
				=A-\alpha uw-\mu_1u+r_2w+r_1v, &(s,t)\in\Omega_\infty,\\
				v_t-Xd_2\Delta_s v-(Yd_2+h'Z+Wd_2)v_s
				=(1-p)\alpha uw-\beta v-r_1v-\mu_2v, &(s,t)\in\Omega_{h_0}, \\
				w_t-Xd_3\Delta_s w-(Yd_3+h'Z+Wd_2)w_s
				=p\alpha uw+\beta v-r_2w-\mu_3w, &(s,t)\in\Omega_{h_0},\ \\
				h'(t)=-\beta v_s(h_0,t)-\mu w_s(h_0,t), &t>0\\
			\end{cases}
		\end{equation}
		with
		\begin{equation} \label{modified free boundary SEIS2}
			\begin{cases}
				u_s(0,t)=v_s(0,t)=w_s(0,t)=0, &t>0, \\
				v(s,t)=w(s,t)=0, &(s,t)\notin\Omega_{h_0}, \\
				u(s,0)=u_0(s), v(s,0)=v_0(s), w(s,0)=w_0(s), &s \geq 0.
			\end{cases}
		\end{equation}
		where $u_0=S_0,\ v_0=E_0,\ w_0=I_0$.
		Let
		\begin{equation*}h^*=-\beta v_0'(h_0)-\mu w_0'(h_0).
		\end{equation*}
		Then for $0<T\leq\frac{h_0}{8(1+h^*)}$,
		define
		\begin{equation*}
			\begin{aligned}
			&H_T :=\big\{h\in C^1([0,T]):h(0)=h_0,\ h'(0)=h^*,\ \lVert h'-h^*\rVert_{C([0,T])}\leq 1\big\},\\
			&U_T :=\big\{u\in C(\overline{\Omega_{\infty T}}):u(s,0)=u_0(s),\ \lVert u-u_0 \rVert_{L^\infty(\overline{\Omega_{\infty T}})}\leq 1\big\},\\
			&V_T :=\big\{v \in C(\overline{\Omega_{\infty T}}):v(s,t)=0 \text{ for } s\geq h_0,\ 0 \leq t \leq T,\\
			&\qquad \qquad v(s,0)=v_0(s) \text{ for } 0\leq s \leq h_0,\ \lVert v-v_0 \rVert _{L^\infty(\overline{\Omega_{\infty T}})} \leq 1\big\},\\
			&W_T :=\big\{w \in C(\overline{\Omega_{\infty T}}):w(s,t)=0 \text{ for } s\geq h_0,\ 0 \leq t \leq T,\\
			&\qquad \qquad \ w(s,0)=w_0(s) \text{ for } 0\leq s \leq h_0,\ \lVert w-w_0 \rVert _{L^\infty(\overline{\Omega_{\infty T}})} \leq 1\big\}.
			\end{aligned}
		\end{equation*}
		Since $h_1(0)=h_2(0)=h_0,$ for $h_1,h_2 \in H_T$,
		\begin{equation} \label{h12}
			\lVert h_1-h_2 \rVert_{C([0,T])} \leq T \lVert h_1'-h_2'\rVert _{C([0,T])}.
		\end{equation}
		Now, define $\Gamma_T$ by
		\begin{equation*}
			\Gamma_T :=U_T \times V_T \times W_T \times H_T.
		\end{equation*}
		Then $\Gamma_T$ is a complete metric space with the metric
		\begin{equation*}
			\begin{aligned}
			&d((u_1,v_1,w_1;h_1),(u_2,v_2,w_2;h_2))\\
			&=\lVert u_1-u_2\rVert_{L^\infty(\overline{\Omega_{\infty T}})}+\lVert v_1-v_2\rVert_{L^\infty(\overline{\Omega_{\infty T}})}+\lVert w_1-w_2\rVert_{L^\infty(\overline{\Omega_{\infty T}})}+\lVert h_1'-h_2'\rVert_{C([0,T])},
			\end{aligned}
		\end{equation*}
		because closed subspace of complete metric space is also complete.
		By using $L^p$ estimate for parabolic equations and Sobolev's embeddings, we can find that for all $(u,v,w;h) \in \Gamma_T$, the following initial boundary value problem

		\begin{equation*}
			\begin{cases}
				\widetilde{u}_t-Xd_1\Delta_s \widetilde{u}-(Yd_1+h'Z+Wd_1) \widetilde{u}_s=A-\alpha uw-\mu_1u+r_2w+r_1v, &(s,t)\in \Omega_\infty,\\
				\widetilde{v}_t-Xd_2\Delta_s \widetilde{v}-(Yd_2+h'Z+Wd_2)\widetilde{v}_s=(1-p)\alpha uw-\beta v-r_1v-\mu_2v, &(s,t)\in\Omega_{h_0},\\
				\widetilde{w}_t-Xd_3\Delta_s \widetilde{w}-(Yd_3+h'Z+Wd_2)\widetilde{w}_s=p\alpha uw+\beta v-r_2w-\mu_3w, &(s,t)\in\Omega_{h_0},\\
				h'(t)=-\beta v_s(h_0,t)-\mu w_s(h_0,t), &t>0,\\
			\end{cases}
		\end{equation*}
		with
		\begin{equation*}
			\begin{cases}
				\widetilde{u}_s(0,t)=\widetilde{v}_s(0,t)=\widetilde{w}_s(0,t)=0, &t>0, \\
				\widetilde{v}(s,t)=\widetilde{w}(s,t)=0, &(s,t)\notin\Omega_{h_0}, \\
				\widetilde{u}(s,0)=u_0(s), \widetilde{v}(s,0)=v_0(s), \widetilde{w}(s,0)=w_0(s), &s \geq 0,
			\end{cases}
		\end{equation*}
		has a unique solution 
		\begin{equation*}
			(\widetilde u,\widetilde v,\widetilde w) \in [C^{1+\gamma}([0,+\infty) \times [0,T])]^3
		\end{equation*}
		and there exists a constant $K$ depending on $\gamma, h_0, \lVert S_0 \rVert, \lVert E_0 \rVert$ and $\lVert I_0 \rVert$ satisfying
		\begin{equation}	\label{tilda uvw estimates}
			\left \{ \begin{aligned}
				&\lVert \widetilde u \rVert _{{C^{1+\gamma}}([0,+\infty) \times [0,T])} \leq K,\\
				&\lVert \widetilde v \rVert _{{C^{1+\gamma}}([0,h_0] \times [0,T])} \leq K,\\
				&\lVert \widetilde w \rVert _{{C^{1+\gamma}}([0,h_0] \times [0,T])} \leq K.
			\end{aligned} \right.
		\end{equation}
		Define
		\begin{equation*}
			\widetilde h(t)=h_0-\int_0^t\beta v_s(h_0,\tau)+\mu w_s(h_0,\tau) \, d\tau.
		\end{equation*}
		Then
		\begin{equation*}
			\begin{aligned}
				&\widetilde h'(t)=-\beta v_s(h_0,\tau)-\mu w_s(h_0,\tau)d\tau,\\
				&\widetilde h(0)=h_0,\\
				&\widetilde h'(0)=-\beta {v_0}'(h_0)-\mu {w_0}'(h_0)=h^*.
			\end{aligned}
		\end{equation*}
		Therefore
		\begin{equation}	\label{h prime estimate}
			\begin{aligned}
			\lVert {\widetilde h}'(t) \rVert _{C^{\frac{\gamma}{2}}([0,T])}&=\lVert \beta v_s(h_0,\tau)+\mu w_s(h_0,\tau) ]\rVert_{C^\frac{\gamma}{2}([0,T])}\leq \beta \lVert v_s \rVert + \mu \lVert w_s \rVert \leq (\beta + \mu)K,
			\end{aligned}
		\end{equation}
		and let 
		\begin{equation*}
			K' \coloneqq (\beta + \mu)K.
		\end{equation*}
		Define $\mathcal{F} : \Gamma_T \to [C([0,+\infty))]^3$ by
		\begin{equation*}
			\mathcal{F}(u(s,t),v(s,t),w(s,t);h(t))=(\widetilde u(s,t),\widetilde v(s,t),\widetilde w(s,t);\widetilde h(t)).
		\end{equation*}
		Then $(u(s,t),v(s,t),w(s,t);h(t))\in \Gamma_T$ is a fixed point of $\mathcal{F}$ if and only if $(u(s,t),v(s,t),w(s,t);h(t))$ is a solution of \cref{modified free boundary SEIS1} and \cref{modified free boundary SEIS2}.
		From 
		\begin{equation*}
			\begin{aligned}
			\frac{1}{T^{\frac{\gamma}{2}}}\underset{t}{\sup}\lvert {\widetilde h}'(t)-h^*(t) \rvert &\leq \frac{1}{T^{\frac{\gamma}{2}}}\underset{t,t'}{\sup}\lvert {\widetilde h}'(t)-{\widetilde h}'(t') \rvert \leq \underset{t,t'}{\sup}\frac{\lvert{\widetilde h}'(t)-{\widetilde h}'(t')\rvert}{\lvert t-t' \rvert^{\frac{\gamma}{2}}}\\
			&=\lVert {\widetilde h}'(t) \rVert _{\dot{C}^{0,\frac{1+\gamma}{2}}([0,T])}\leq \lVert {\widetilde h}'(t) \rVert_{C^\frac{\gamma}{2}([0,T])},
			\end{aligned}
		\end{equation*}
		together with \cref{tilda uvw estimates} and \cref{h prime estimate}, we have
		\begin{equation} \label{h' holder}
			\lVert {\widetilde h}'-h^* \rVert_{C([0,T])} \leq \lVert {\widetilde h}' \rVert_{C^{\frac{\gamma}{2}}([0,T])}T^{\frac{\gamma}{2}}\leq(\beta + \mu)KT^{\frac{\gamma}{2}}.
		\end{equation}
		Similarly, calculating the norm of $\tilde u-u_0, \tilde v-v_0 \text{ and }\tilde w-w_0$,
		\begin{equation} \label{u holder}
			\begin{aligned}
				\lVert \widetilde u-u_0 \rVert_{C(\overline{\Omega_{\infty T}})}&\leq \underset{x,t}{\sup}\lvert \widetilde u(x,t) -u_0(x)\rvert T^{\frac{1+\gamma}{2}}\leq \underset{x,x',t,t'}{\sup}\lvert \widetilde u(x,t)-\widetilde u(x',t')\rvert T^{\frac{1+\gamma}{2}}\\
				&\leq\lVert \widetilde u \rVert_{\dot{C}^{0,\frac{1+\gamma}{2}}(\overline{\Omega_{\infty T}})}T^{\frac{1+\gamma}{2}}\leq \lVert \widetilde u \rVert_{C^{0,\frac{1+\gamma}{2}}(\overline{\Omega_{\infty T}})}T^{\frac{1+\gamma}{2}} \leq \lVert \widetilde u \rVert_{C^{1+\gamma}(\overline{\Omega_{\infty T}})}T^{\frac{1+\gamma}{2}},
			\end{aligned}
		\end{equation}
		\begin{equation} \label{v holder}
			\begin{aligned}
				\lVert \widetilde v-v_0 \rVert_{C(\overline{\Omega_{h_0 T}})}&\leq \underset{x,t}{\sup}\lvert \widetilde v(x,t) -v_0(x)\rvert T^{\frac{1+\gamma}{2}}\leq \underset{x,x',t,t'}{\sup}\lvert \widetilde v(x,t)-\widetilde v(x',t')\rvert T^{\frac{1+\gamma}{2}}\\
				&\leq\lVert \widetilde v \rVert_{\dot{C}^{0,\frac{1+\gamma}{2}}(\overline{\Omega_{h_0 T}})}T^{\frac{1+\gamma}{2}}\leq \lVert \widetilde v \rVert_{C^{0,\frac{1+\gamma}{2}}(\overline{\Omega_{h_0 T}})}T^{\frac{1+\gamma}{2}}\leq \lVert \widetilde v \rVert_{C^{1+\gamma}(\overline{\Omega_{h_0 T}})}T^{\frac{1+\gamma}{2}},
			\end{aligned}
		\end{equation}
		and
				\begin{equation} \label{w holder}
			\begin{aligned}
				\lVert \widetilde w-w_0 \rVert_{C(\overline{\Omega_{h_0 T}})}&\leq \underset{x,t}{\sup}\lvert \tilde w(x,t) -w_0(x)\rvert T^{\frac{1+\gamma}{2}}\leq \underset{x,x',t,t'}{\sup}\lvert \widetilde w(x,t)-\widetilde w(x',t')\rvert T^{\frac{1+\gamma}{2}}\\
				&\leq \lVert \widetilde w \rVert_{\dot{C}^{0,\frac{1+\gamma}{2}}(\overline{\Omega_{h_0 T}})}T^{\frac{1+\gamma}{2}}\leq \lVert \widetilde w \rVert_{C^{0,\frac{1+\gamma}{2}}(\overline{\Omega_{h_0 T}})}T^{\frac{1+\gamma}{2}}\leq \lVert \tilde w \rVert_{C^{1+\gamma}(\overline{\Omega_{h_0 T}})}T^{\frac{1+\gamma}{2}}.
			\end{aligned}
		\end{equation}
		Therefore 
		\begin{equation}
		\lVert \widetilde u-u_0 \rVert_{C(\overline{\Omega_{\infty T}})}, \lVert \widetilde v-v_0 \rVert_{C(\overline{\Omega_{\infty T}})}, \lVert \widetilde w-w_0 \rVert_{C(\overline{\Omega_{\infty T}})} \leq KT^{\frac{1+\gamma}{2}},	
		\end{equation}
and if we take 
\begin{equation*}
	T \leq \min \{((\beta+\mu)K)^{-\frac{2}{\gamma}},K^{-\frac{2}{1+\gamma}}\},
\end{equation*}
then $\mathcal{F}$ maps $\Gamma_T$ to itself. To show $\mathcal{F}$ is a contraction map on $\Gamma_T$ for sufficiently small $T>0$, let 
\begin{equation*}
	(u_1,v_1,w_1;h_1),(u_2,v_2,w_2;h_2)\in \Gamma_T,
\end{equation*}
and denote
\begin{equation*}
	(\widetilde u_1(s,t),\widetilde v_1(s,t),\widetilde w_1(s,t);\widetilde h_1(t))=\mathcal{F}(u_1(s,t),v_1(s,t),w_1(s,t);h_1(t))
\end{equation*}  and 
\begin{equation*}
	(\widetilde u_2(s,t),\widetilde v_2(s,t),\widetilde w_2(s,t);\widetilde h_2(t))=\mathcal{F}(u_2(s,t),v_2(s,t),w_2(s,t);h_2(t)).
\end{equation*}
Then it follows from \cref{h' holder}, \cref{u holder}, \cref{v holder}, and \cref{w holder} that 
		\begin{equation*}
			\begin{aligned}
				&\lVert \widetilde u_i \rVert_{C^{1+\gamma}(\overline{\Omega_{\infty T}})} \leq K,\\
				&\lVert \widetilde v_i \rVert_{C^{1+\gamma}(\overline{\Omega_{h_0 T}})} \leq K,\\
				&\lVert \widetilde w_i \rVert_{C^{1+\gamma}(\overline{\Omega_{h_0 T}})} \leq K,\\
				&\lVert {\widetilde h_i}' \rVert_{C^{\frac{\gamma}{2}}([0,T])} \leq K',\\
			\end{aligned}
		\end{equation*}
		for $i=1,2$. Then if setting $U=\widetilde u_1 -\widetilde u_2$, we find that $U(s,t)$ satisfies
		\begin{equation*}
			\begin{aligned}
				&U_t-X(h_2,s)d_1\Delta_s U-(Y(h_2,s)d_1+h_2'Z(h_2,s)+W(h_2,s)d1)U_s\\
				&=(X(h_1,s)-X(h_2,s))d_1 \Delta_s \widetilde u_1 + (Y(h_1,s)-Y(h_2,s))d_1U_s\\
				&+(h_1'Z(h_1,s)+h_2'Z(h_2,s))U_s+W(h_1,s)-W(h_2,s))d_1U_s\\
				&+\alpha(u_2w_2-u_1w_1)+\mu_1(u_2-u_1)+r_2(w_1-w_2)+r_1(v_1-v_2),
			\end{aligned}
		\end{equation*}
		for $(s,t)\in\Omega_{h_0},$ and
		\begin{equation*}
			\begin{aligned}
				&\frac{\partial U}{\partial s}(t,0)=0, &&\quad t>0,\\
				&U(t,h_0)=0, &&\quad t>0,\\
				&U(0,s)=0, &&\quad s\leq h_0.
			\end{aligned}
		\end{equation*}
		Using the $L^p$ estimates for parabolic equations and Sobolev's imbedding, we obtain
		\begin{equation} \label{u tilde}
			\lVert \widetilde u_1 - \widetilde u_2 \rVert_{C^{1+\gamma}(\overline{\Omega_{h_0 T}})} \leq K''(\lVert u_1 - u_2 \rVert_{C(\overline{\Omega_{h_0 T}})}+\lVert h_1 -h_2 \rVert_{C^1([0,T])},
		\end{equation}
		where $K''$ depends on $K, K'$ and the functions $X, Y, Z,$ and $W$. Similarly, 
		\begin{equation} \label{vw tilde}
			\begin{aligned}
				&\lVert \widetilde v_1 - \widetilde v_2 \rVert_{C^{1+\gamma}(\overline{\Omega_{h_0 T}})} \leq K''(\lVert v_1 - v_2 \rVert_{C(\overline{\Omega_{h_0 T}})}+\lVert h_1 -h_2 \rVert_{C^1([0,T])},\\
				&\lVert \widetilde w_1 - \widetilde w_2 \rVert_{C^{1+\gamma}(\overline{\Omega_{h_0 T}})} \leq K''(\lVert w_1 - w_2 \rVert_{C(\overline{\Omega_{h_0 T}})}+\lVert h_1 -h_2 \rVert_{C^1([0,T])},
			\end{aligned}
		\end{equation}
		for same $K''$ depending on $K, K'$ and the functions $X, Y, Z,$ and $W$. Taking the difference of the equations for $\widetilde h_1$ and $\widetilde h_2$ results in
		\begin{equation} \label{h tilde}
			\lVert \widetilde h_1'-\widetilde h_2' \rVert_{C^{\frac{\gamma}{2}}([0,T])}\leq \mu \lVert \widetilde w_{1,s} - \widetilde w_{2,s} \rVert_{C_t^{\frac{\gamma}{2}}(\overline{\Omega_{h_0 T}})} + \beta \lVert \widetilde v_{1,s} - \widetilde v_{2,s} \rVert_{C_t^{\frac{\gamma}{2}}(\overline{\Omega_{h_0 T}})}.
		\end{equation}
		Therefore assuming $T \leq 1$ with \cref{h12}, \cref{u tilde}, \cref{vw tilde}, and \cref{h tilde}, we obtain
		\begin{equation*}
		\begin{aligned}
			&\lVert \widetilde u_1 - \widetilde u_2 \rVert_{C^{1+\gamma}(\overline{\Omega_{h_0 T}})}
			+\lVert \widetilde v_1 - \widetilde v_2 \rVert_{C^{1+\gamma}(\overline{\Omega_{h_0 T}})}
			+\lVert \widetilde w_1 - \widetilde w_2 \rVert_{C^{1+\gamma}(\overline{\Omega_{h_0 T}})}
			+\lVert \widetilde h_1'-\widetilde h_2' \rVert_{C^{\frac{\alpha}{2}}([0,T])}\\
			&\leq K''' (\lVert u_1-u_2 \rVert_{C(\overline{\Omega_{h_0 T}})}
			+\lVert v_1-v_2 \rVert_{C(\overline{\Omega_{h_0 T}})}
			+\lVert w_1-w_2 \rVert_{C(\overline{\Omega_{h_0 T}})}
			+\lVert h_1'-h_2' \rVert_{C([0,T])})
		\end{aligned}
		\end{equation*}
		with $K'''$ depending on $K'', \mu$ and $\beta$. Therefore for sufficiently small $T$ $\mathcal{F}$ is a contraction mapping on $\Gamma_T$, so $\mathcal{F}$ has a unique fixed point $(u,v,w;h)$ in $\Gamma_T$, and it is a unique solution for the systems.
		% 정확하게 다시 계산 uniqueness 보이기 %
%		\begin{equation*}
%			\begin{aligned}
%				&U_t-Xd_1\Delta_sU-(Yd_1+h'Z+Wd_1)U_s\\
%				&=\tilde u_1(t)-X(h_2,s)d_1\Delta_s \tilde u_1-Y(h_2,s)d_1 \tilde u_{1,s}-h_2'Z(h_2,s)\tilde u_{1,s}-W(h_2,s)\tilde u_{1,s}\\
%				&=[X(h_1,s)-X(h_2,s)]d_1\Delta_s \tilde u_{1}+[Y(h_1,s)-Y(h_2,s)]d\tilde u_{1,s}+[h'_1Z(h_1,s)-h'_2Z(h_2,s)]d\tilde v
%			\end{aligned}
%		\end{equation*}
\end{proof}

Using Schauder estimates, we conclude that the solution achieves sufficient regularity to qualify as a classical solution.

\begin{proposition}
	A unique solution obtained in \cref{shorttime} is a classical solution.
\end{proposition}

\begin{proof}
Since $\left(u\left(s,t\right),v\left(s,t\right),w\left(s,t\right);h\left(t\right)\right) \in C^{1+\gamma}(\overline{\Omega_{\infty T}})\times C^{1+\gamma}(\overline{\Omega_{h_0 T}})^2\times C^{\frac{\gamma}{2}}([0,T])$ and each $u,\ v,\ w$, and $h$ is at left side of the equation \cref{free boundary SEIS} and with \cref{initial cond}, we have additional regularity for $(u(s,t),v(s,t),w(s,t);h(t))$ as a solution of the problem \cref{modified free boundary SEIS1} with \cref{modified free boundary SEIS2} by using Schauder estimates, namely, 
		\begin{equation*}
			\begin{aligned}
				h &\in C^{1+\frac{\gamma}{2}}([0,T]),\\
				S &\in C^{2+\gamma}(\Omega_{\infty T}),\\
				\text{and} \quad E,I &\in C^{2+\gamma}(\Omega_{hT}).
			\end{aligned}
		\end{equation*}
		with estimates \cref{tilda uvw estimates} and \cref{h prime estimate}. Thus, we can conclude that ($S(r,t)$, $E(r,t)$, $I(r,t)$; $h(t)$) is the classical solution of the problem \cref{free boundary SEIS} with \cref{initial cond}.
	\end{proof}
	
\subsection{Global Well--posedness}
	We can improve the existence of the solution to long-time existence if the problem \cref{free boundary SEIS} with \cref{initial cond} has the same diffusion rate, i.e., we will show that with assuming $d_1=d_2=d_3=d$ leads to this result. This assumption is necessary because a standard recursive method for using the comparison principle of each component cannot be applied to the system. The structure that expresses reinfection causes difficulty in using the recursive method, and this is the main point of our proof.
	
	\begin{theorem} \label{longtime}
		If $d_1=d_2=d_3=d$ for some $d$, the solution of the problem \cref{free boundary SEIS} with \cref{initial cond} exists and is unique for all $t \in (0,\infty)$.
	\end{theorem}
	
Before proving the theorem, it is necessary to establish uniform estimates for each compartment $S$, $E$, and $I$ and the free boundary $h$, which will be demonstrated first.
	
	\begin{lemma} \label{uniform SEI}
		Suppose that $d_1=d_2=d_3=d$ for some $d$, and let $(S,E,I;h)$ be a bounded solution of the problem  \cref{free boundary SEIS} with \cref{initial cond} for $t \in (0,T_0)$ for some $T_0 \in (0,\infty]$. Then there exist $C_1$ and $C_2$ independent of $T_0$ such that
	\begin{equation*}
		\left \{ \begin{aligned}
				&0<S(r,t)\leq C_1, \qquad &&(r,t) \in \bar\Omega_{\infty T_0},\\
				&0<I(r,t),\ E(r,t)\leq C_2, &&(r,t) \in \bar\Omega_{h T_0}.
			\end{aligned} \right.
	\end{equation*}
	\end{lemma}
	
	\begin{proof}
		Let $(S,E,I;h)$ be a bounded solution to the problem \cref{free boundary SEIS} with \cref{initial cond} for $t\in(0,T_0)$ for some $T_0\in(0,\infty]$ by \cref{shorttime}.
		Given $S(r,t), E(r,t), I(r,t) \geq 0$ for $(r,t) \in \bar\Omega_{\infty T_0}$ as long as the solution exists, applying the strong maximum principle to the equations in $\bar \Omega _{h T_0}$ implies that $S(r,t), E(r,t), I(r,t) >0$ for $(r,t)\in \bar \Omega_{h T_0}$, as the parabolic boundary is $[0,h(t)] \times \{ 0\}\  \cup \ \{ 0 \} \times (0,T_0) \ \cup \ \{h(t)\} \times (0,T_0)$.
		Since $d_1=d_2=d_3=d$ by assumption, 
		\begin{equation*}
		\begin{aligned}
			S_t+E_t+I_t-d\Delta S-d\Delta E-d\Delta I&=A-\mu_1 S-\mu_2 E-\mu_3 I\leq A-\min\{\mu_1, \mu_2, \mu_3 \}(S+E+I).
		\end{aligned}
		\end{equation*}
		Thus for $X(r,t):=S(r,t)+E(r,t)+I(r,t)$ and $b:=\min\{\mu_1, \mu_2, \mu_3 \}$,
		\begin{equation*}
			X_t+d\Delta X \leq A-bX.
		\end{equation*}
		Since $\overline X=\frac{A}{b}(\lVert S_0+E_0+I_0\rVert_C-\frac{A}{b})e^{-bt}$ is the solution of the problem,
		\begin{equation*}
			\left \{ \begin{aligned}
				&\frac{d\overline{X}}{dt}=A-b\overline{X},\qquad t>0,\\
				&\overline{X}(0)=\lVert S_0+E_0+I_0\rVert_C,
			\end{aligned} \right.
		\end{equation*}
		by comparison principle,
		\begin{equation*}
			X(r,t)=S(r,t)+E(r,t)+I(r,t) \leq \overline{X}(t)
		\end{equation*}
		for $r\geq 0$ and $t \in (0,\infty)$.
		Since $\displaystyle \lim_{t\to \infty}X=\frac{A}{b}$,
		\begin{equation*}
			\limsup_{t \to \infty}S(r,t)+E(r,t)+I(r,t) \leq \frac{A}{b}
		\end{equation*}
		uniformly for $r\in [0,\infty)$.
		Thus we can deduce that each $S(r,t)$, $E(r,t)$, and $I(r,t)$ is uniformly bounded, in other words, there exists $C_1,\ C_2>0$ such that
		\begin{equation*}
			\left \{ \begin{aligned}
				&0<S(r,t)\leq C_1 \qquad &&(r,t) \in \Omega_\infty,\\
				&0<I(r,t),\ E(r,t)\leq C_2, &&(r,t) \in \Omega_h.
			\end{aligned} \right.
		\end{equation*}
	\end{proof}
	
	\begin{lemma} \label{uniform h'}
		Suppose that $d_1=d_2=d_3=d$ for some $d$, and let $(S,E,I;h)$ be a bounded solution of the problem \cref{free boundary SEIS} with \cref{initial cond} for $t \in (0,T_0)$ for some $T_0 \in (0,\infty]$. Then there exists a constant $C_3$ independent of $T_0$ such that
		\begin{equation*}
			0<h'(t)\leq C_3
		\end{equation*}
		for $t\in (0,T_0)$.
	\end{lemma}
	
	\begin{proof}
		By Hopf's lemma to the equation for $E$ and $I$ of \cref{free boundary SEIS}, we can know that $E_r(h(t),t)<0$ and $I_r(h(t),t)<0$ for $0<t<T_0$.
		Then
		\begin{equation*}
			h'(t)=-\beta E_r(h(t),t)-\mu I_r(h(t),t)>0,
		\end{equation*}
		for $0<t<T_0$.
		Let
		\begin{equation*}
			\Omega_M \coloneqq\{(r,t):h(t)-\frac{1}{M}<r<h(t),\ 0<t<T_0\}
		\end{equation*}
		and construct an auxiliary function
		\begin{equation*}
			w(r,t) \coloneqq C_2(2M(h(t)-r)-M^2(h(t)-r)^2).
		\end{equation*}
		for some $M>0$. Then for $(r,t) \in \Omega_M$,
		\begin{equation*}
		\begin{aligned}
			&w_t=2C_2Mh'(t)(1-M(h(t)-r))\\
			&w_r=2C_2M(-1+M(h(t)-r))\\
			&w_{rr}=-2C_2M^2\\
			&\Delta w=-2C_2M^2+\frac{n-1}{r}2C_2M(-1+M(h(t)-r))\leq-2C_2M^2.
		\end{aligned}
		\end{equation*}
		Since $(1-p)\alpha IS-(\beta_1+r_1+\mu_2)E \leq (1-p)\alpha C_1 C_2$, we can calculate
		\begin{equation*}
			w_t-d_2w \geq 2d_2C_2M^2 \geq (1-p) \alpha C_1 C_2
		\end{equation*}
		if $M$ satisfies that
		\begin{equation*}
			M^2 \geq \frac{(1-p)\alpha C_1}{2d_2}.
		\end{equation*}
		Otherwise,
		\begin{equation*}
			w\left(h(t)-\frac{1}{M},t\right)=C_2 \geq E\left(h(t)-\frac{1}{M},t\right)
		\end{equation*}
		and
		\begin{equation*}
			w(h(t),t)=0=E(h(t),t).
		\end{equation*}
		If we choose $M$ satisfying $E_0(r)\leq w(r,0)$ for $r \in [h_0-\frac{1}{M},h_0]$, then by maximum principle, $E(r,t) \leq w(r,t)$ for $(r,t) \in \Omega_M$.
		Then we can also obtain that 
		\begin{equation}
			E_r(h(t),t) \geq w_r(h(t),t)=2C_2M.
		\end{equation}
		Similarly, for $(r,t) \in \Omega_M$,
		\begin{equation*}
			\begin{aligned}
				&w_t=2C_2Mh'(t)(1-M(h(t)-r)) \geq 0,\\
				&\Delta w\leq -2C_2M^2.
			\end{aligned}
		\end{equation*}
		Since $p\alpha IS +\beta_1 E-(r_2+\mu_3)I\leq p \alpha C_1 C_2+\beta_1 C_2$, we can calculate
		\begin{equation*}
			w_t-d_3\Delta w \geq 2d_3C_2M^2 \geq p \alpha C_1 C_2 + \beta_1 C_2
		\end{equation*}
		if $M$ satisfies that
		\begin{equation*}
			M^2 \geq \frac{p \alpha C_1 +\beta _1}{2 d_3}.
		\end{equation*}
		Otherwise,
		\begin{equation*}
			w\left(h(t)-\frac{1}{M},t\right)=C_2 \geq I\left(h(t)-\frac{1}{M},t\right),
		\end{equation*}
		and
		\begin{equation*}
			w(h(t),t)=0=I(h(t),t).
		\end{equation*}
		If we choose $M$ satisfying $I_0(r) \leq w(r,0)$ for $r \in [h_0-\frac{1}{M},h_0]$, then by maximum principle, $I(r,t) \leq w(r,t)$ for $(r,t) \in \Omega_M$.
		Then
		\begin{equation*}
			h'(t)=-\beta E_r(h(t),t)-\mu I_r(h(t),t) \leq 2(\beta+\mu)C_2M \eqqcolon C_3.
		\end{equation*}
		Therefore we have to find some $M$ independent of $T_0$ such that $E_0(r) \leq w(r,0)$ and $I_0(r) \leq w(r,0)$ for $r\in [h_0-\frac{1}{M},h_0]$.
		Since
		\begin{equation*}
			w_r(r,0)=2C_2M(-1 +M(h_0-r)) \leq -C_2M
		\end{equation*}
		for $r \in [h_0-\frac{1}{2M}, h_0]$ and $I_0'(r) \leq 0$ for $r \in (0,h_0]$,
		\begin{equation*}
			w_r(r,0) \leq -C_2M \leq -\frac{4}{3} \lVert I_0 \rVert _{C^1} \leq {I_0}'(r)
		\end{equation*}
		for $r \in [h_0-\frac{1}{2M}, h_0]$ if $M=\max \left\{ 
		\sqrt{\frac{(1-p)\alpha C_1}{2d_2}}, \sqrt{\frac{p\alpha C_1+\beta_1}{2d_3}},\frac{4\lVert I_0 \rVert _{C^1}}{3C_2} \right\}$. Moreover, we can deduce that $w(r,0) \geq I_0(r)$ for $r \in [h_0-\frac{1}{2M},h_0]$, since $w(h_0,0)=I_0(h_0)=0$. Furthermore, if $M \geq 1$,
		\begin{equation*}
			w(r,0)=C_2(2M(h_0-r)-M^2(h_0-r)^2) \leq w\left(h_0-\frac{1}{2M},0\right)=\frac{3}{4} C_2
		\end{equation*}
		and
		\begin{equation*}
			I_0(r) \leq \lVert I_0 \rVert _{C^1} \frac{1}{M} \leq \frac{3}{4} C_2.
		\end{equation*}
		Therefore $I_0(r) \leq w(r,0)$ for $r \in [h_0-\frac{1}{M},h_0]$, so we can conclude that
		\begin{equation*}
			\begin{aligned}
				&I_r(h(t),t) \geq w_r(h(t),t)=-2C_2M\\
				&E_r(h(t),t) \geq w_r(h(t),t)=-2C_2M,
			\end{aligned}
		\end{equation*}
		so
		\begin{equation*}
			h'(t) \leq (\mu + \beta) 2C_2 M \eqqcolon C_3.
		\end{equation*}
	\end{proof}
	
	\begin{proof}[proof of Theorem 3.4]
		
		By the uniqueness of the solution, there exists $T_{\max}$ such that $[0,T_{\max})$ is the maximal time interval where the solution exists.
		Suppose that $T_{\max}<\infty$. Since there exist $C_1$, $C_2$, and $C_3$ independent of $T_{\max}$ such that
		\begin{equation*}
			\begin{aligned}
				&0\leq S(r,t) \leq C_1, \qquad &&(r,t) \in [0,\infty) \times [0, T_{\max}) \\
				&0 \leq E(r,t), I(r,t) \leq C_2, &&(r,t) \in [0,h(t)] \times [0, T_{\max}) \\
				&0 \leq h'(t) \leq C_3, &&t \in [0, T_{\max}), \\
				&h_0 \leq h(t) \leq h_0+C_3t, &&t \in [0,T_{\max}),
			\end{aligned}
		\end{equation*}
		for fixed $\delta_0 \in (0,T_{\max})$ and $M>T_{\max}$, by , there exists $C_4>0$ depending on $\delta_0, M, C_1, C_2$ and $C_3$ such that
		\begin{equation*}
			\lVert S(\cdot,t)\rVert_{C^2[0,\infty)},\lVert E(\cdot,t)\rVert_{C^2[0,h(t)]},\lVert I(\cdot,t)\rVert_{C^2[0,h(t)]} \leq C_4
		\end{equation*}
		for $t\in [\delta_0,T_{\max})$.
		By the proof of \cref{shorttime}, there exists $\tau>0$ depending only on $C_1, C_2, C_3$ and $C_4$ such that the solution of the problem \cref{free boundary SEIS} with \cref{initial cond} with initial time $T_{\max}-\frac{\tau}{2}$ can be extended uniquely to the time $T_{\max}+\frac{\tau}{2}$.
		It is a contradiction, so we can conclude that $T_{\max}=\infty$.
		Therefore we can say that there exists the unique solution of the problem \cref{free boundary SEIS} with \cref{initial cond}.
	\end{proof}
	
\section{Asymptotic Behavior of the Model}
Most mathematical biologists are primarily concerned with the asymptotic behavior of solutions as time approaches infinity. Accordingly, we investigate this behavior after establishing the well-posedness of our model. In general, once a solution representing the spread of the infected individuals  is obtained, it is natural to ask how far the infection will propagate or whether it will eventually die out.

\subsection{The Basic Reproductive Number}
In epidemiology, the reproductive number denoted $\mathcal R_0$ is the average number of cases directly generated by one typical infected individual in a population consisting of the susceptibles only. It is among the quantities most urgently estimated for infectious diseases in outbreak situations, and its value provides insight when designing control interventions for established infections. Therefore reproductive number $\mathcal R_0$ serves as a threshold parameter that determines vanishing spreading dichotomy phenomenon.

We will be interested primarily in equilibrium solutions of the problem \cref{free boundary SEIS} with \cref{initial cond}, i.e., solutions of the problem
\begin{equation} \label{SEIS equilibrium1}
		\left \{ \begin{aligned}	
		&d_1\Delta S+A-\alpha IS-\mu_1 S+r_2 I+r_1E=0,&&(r,t)\in\Omega_\infty,\\
		&d_2\Delta E+(1-p)\alpha IS-\beta_1 E-r_1E-\mu_2 E=0,&&(r,t)\in \Omega_h,\\
		&d_3\Delta I+p\alpha IS+\beta_1 E-r_2I-\mu_3 I=0, &&(r,t)\in \Omega_h,\\
		&h'(t)=-\beta E_r(h(t),t)-\mu I_r(h(t),t),&&t>0,\\
		&E(r,t)=I(r,t)=0, &&(r,t)\notin \Omega_h,\\
		&S_r(0,t)=E_r(0,t)=I_r(0,t)=0, &&t>0,
		\end{aligned} \right.
\end{equation}
with initial conditions \cref{initial cond}.

Then we are only interested in solutions $(S,E,I;h)$ of \cref{SEIS equilibrium1} with \cref{initial cond} which satisfy $E \geq 0$ and $I \geq 0$ for $r<h$. Furthermore, in epidemiology, a disease-free equilibrium is a solution where $I=0$ for any $r<h$. Otherwise, an endemic equilibrium is a solution where $I>0$ for some $r<h$. Shortly, we will refer to them as DFE and EE in the remaining parts of this paper.

Most papers\cite{Yihong_2010, Allen_2008, Diekmann_2013} revealed that $R_0$ affects to the existence of DFE or EE and their stability. Morover, the reproductive number $\mathcal{R}_0$ is defined by using DFE and EE classicaly.

\begin{definition}
	Define reproductive number $\mathcal{R}_0>0$ such that, if $\mathcal{R}_0<1$, then the DFE is locally asymptotically stable, and the disease cannot invade the population, but if $\mathcal{R}_0>1$, then the DFE is unstable and invasion is always possible.
\end{definition}

Thus, by using the definition of the basic reproductive number $\mathcal R_0$, we will proceed in a similar way to show the asymptotic behavior of the model. Our goal is to predict the number based on known results and show that the behavior of the solution satisfies the conditions of the definition of the reproductive number $\mathcal R_0$.
\subsection{Vanishing Case $\mathcal R_0<1$}
We suppose that $d_1=d_2=d_3=d$ for some $d>0$ to use the result of the previous theorems.
\begin{proposition}\label{theorem 4.2}
	If $d_1=d_2=d_3=d$, and $\mathcal R_0^{\max}(:=\frac{\alpha A}{\mu_1 \min \{ r_1+\mu_2,\ r_2+\mu_3\}})<1$, then \[\lim_{t \to \infty}\lVert E(\cdot,t) \rVert_{C([0,T])}=\lim_{t \to \infty}\lVert I(\cdot,t) \rVert_{C([0,T])}=0,\] and $h_\infty<\infty$. Moreover, $\lim_{t \to \infty}S(r,t)=\frac{A}{\mu_1}$ uniformly in any bounded subset of $[0,\infty)$.
\end{proposition}
\begin{proof}
	By the proof of \cref{longtime}, we know that
	\begin{equation*}
		S(r,t)\leq C_1 \leq \frac{A}{\mu_1},
	\end{equation*}
	since $b=\min{\mu_1, \mu_2, \mu_3}$, $\mu_1 \leq \mu_2$ and $\mu_1 \leq \mu_3$. For $R_0<1$, there exists $T_0$ such that
	\begin{equation*}
		S(r,t) \leq \frac{A}{\mu_1}\frac{1+\mathcal R_0}{2\mathcal R_0}
	\end{equation*}
	in $[0,\infty) \times [T_0,\infty)$. Since
	\begin{equation*}
		\begin{aligned}
			(E+I)_t-d\Delta (E+I)&=\alpha IS-(r_1+\mu_2)E-(r_2+\mu_3)I \leq\alpha IS -\min \{ r_1+\mu_2,\ r_2+\mu_3\}(E+I)\\
			&\leq\left( \alpha \frac{A}{\mu_1} \frac{1+R_0}{2R_0}-\min\{ r_1+\mu_2,\ r_2+\mu_3\}\right)(E+I)
		\end{aligned}
	\end{equation*}
	for $(r,t) \in \Lambda_h$, let $c=\min \{ r_1+\mu_2,\ r_2+\mu_3\}$. Then $E(r,t)$ and $I(r,t)$ satisfy
	\begin{equation*}
		\left \{ \begin{aligned}
			&(E+I)_t-d\Delta (E+I)\leq\left( \alpha \frac{A}{\mu_1} \frac{1+R_0}{2R_0}-c\right)(E+I),\quad &&(r,t)\in \Lambda_h,\\
			&E+I=0, &&(r,t) \in \Gamma_h\\
			&(E+I)_r(0,t)=0, &&t>0,\\
			&(E+I)(r,T_0)>0, &&0\leq r \leq h(T_0).
		\end{aligned} \right.
	\end{equation*}
	If $\frac{\alpha A}{\mu_1c}<1$, $\lVert E+I \rVert_C \to 0$ as $t \to \infty$. Then it implies that $\lVert E \rVert_C \to 0$ and $\lVert I \rVert_C \to 0$. Next, we want to show that $h_\infty < \infty$.
	\begin{equation*}
	\begin{aligned}
		&\frac{d}{dt} \int_0^{h(t)}r^{n-1}(E(r,t)+I(r,t))dr\\
		&=\int_0^{h(t)}r^{n-1}(E_t(r,t)+I_t(r,t))dr+h'(t)h^{n-1}\left(E(h(t),t)+I(h(t),t)\right)\\
		&=\int_0^{h(t)}r^{n-1}(E_t(r,t)+I_t(r,t))dr\\
		&=\int_0^{h(t)}dr^{n-1}(\Delta E+\Delta I)dr+\int_0^{h(t)}f(r,t)r^{n-1}dr\\
		&=\int_0^{h(t)}d(r^{n-1}(E+I)_r)_rdr+\int_0^{h(t)}f(r,t)r^{n-1}dr\\
		&=dh^{n-1}(t)(E+I)_r(h(t),t))+\int_0^{h(t)}f(r,t)r^{n-1}dr\\
		&=dh^{n-1}(t)\left(-\frac{1}{\beta}h'(t)+\left(1-\frac{\mu}{\beta}\right)I_r\right)+\int_0^{h(t)}f(r,t)r^{n-1}dr\\
		&=-\frac{d}{\beta}h^{n-1}(t)h'(t)+d\left(1-\frac{\mu}{\beta}\right)h^{n-1}(t)I_r(h(t),t)+\int_0^{h(t)}f(r,t)r^{n-1}dr
	\end{aligned}
	\end{equation*}
	where $f=\alpha I(r,t)S(r,t)-(r_1+\mu_2)E(r,t)-(r_2+\mu_3)I(r,t)$.
	Then we continue to calculate after dividing to two cases. If $1-\frac{\mu}{\beta} \geq 0$,
	\begin{equation*}
		\frac{d}{dt} \int_0^{h(t)}r^{n-1}(E(r,t)+I(r,t))dr\leq-\frac{d}{\beta}h^{n-1}(t)h'(t)+\int_0^{h(t)}f(r,t)r^{n-1}dr.
	\end{equation*}
	Thus for $T_1>T_0$, if we take integral from $T_0$ to $T_1$,
	\begin{equation*}
		\begin{aligned}
			&\int_0^{h(T_1)}r^{n-1}(E(r,T_1)+I(r,T_1))dr-\int_0^{h(T_0)}r^{n-1}(E(r,T_0)+I(r,T_0))dr\\
			&\leq-\frac{1}{n}\frac{d}{\beta}h^n(T_1)+\frac{1}{n}\frac{d}{\beta}h^n(T_0)+\int_{T_0}^{T_1}\int0^{h(T_0)}r^{n-1}(E(r,T_0)+I(r,T_0))drdt.
		\end{aligned}
	\end{equation*}
	Since $0<S<\frac{A}{\mu_1}\frac{1+R_0}{2R_0}$ for all $r \in \Omega_h$ and $t\geq T_0$,
	\begin{equation*}
		\alpha S-\min\{r_1+\mu_2,\ r_2+\mu_3\}\leq0 \ \text{for } t\geq T_0,
	\end{equation*}
	so
	\begin{equation*}
		\begin{aligned}
			\frac{d}{\beta}h^n(T_1)\leq&-\int_0^{h(T_1)}r^{n-1}(E(r,T_1)+I(r,T_1))dr+\int_0^{h(T_0)}r^{n-1}(E(r,T_0)+I(r,T_0))dr+\frac{d}{\beta}h^n(T_0).
		\end{aligned}
	\end{equation*}
	Then taking limit to $T_1$ implies that $\displaystyle h_\infty\coloneqq \lim_{t \to \infty}<\infty$. Otherwise, i.e., $1-\frac{\mu}{\beta} < 0$, since $h'(t)=\mu I_r-\beta E_r \geq -\mu I_r$,
	\begin{equation*}
		\frac{d}{dt} \int_0^{h(t)}r^{n-1}(E(r,t)+I(r,t))dr \leq -\frac{d}{\mu} h^{n-1}(t)h'(t)+\int_0^{h(t)}f(r,t)r^{n-1}dr.
	\end{equation*}
	Similarly, for $T_1>T_0$, if we take integral from $T_0$ to $T_1$,
	\begin{equation*}
		\begin{aligned}
			&\int_0^{h(T_1)}r^{n-1}(E(r,T_1)+I(r,T_1))dr-\int_0^{h(T_0)}r^{n-1}(E(r,T_0)+I(r,T_0))dr\\
			&\leq-\frac{1}{n}\frac{d}{\mu}h^n(T_1)+\frac{1}{n}\frac{d}{\mu}h^n(T_0)+\int_{T_0}^{T_1}\int0^{h(T_0)}r^{n-1}(E(r,T_0)+I(r,T_0))drdt.
		\end{aligned}
	\end{equation*}
	Since $0<S<\frac{A}{\mu_1}\frac{1+R_0}{2R_0}$ for all $(r,t) \in \Lambda_h$,
	\begin{equation*}
		\alpha S-\min\{r_1+\mu_2,\ r_2+\mu_3\}\leq0 \ \text{for } t\geq T_0,
	\end{equation*}
	so
	\begin{equation*}
		\begin{aligned}
			\frac{d}{\mu}h^n(T_1)\leq&-\int_0^{h(T_1)}r^{n-1}(E(r,T_1)+I(r,T_1))dr+\int_0^{h(T_0)}r^{n-1}(E(r,T_0)+I(r,T_0))dr+\frac{d}{\mu}h^n(T_0).
		\end{aligned}
	\end{equation*}
	By taking the limit to $T_1$, we can complete the proof. Moreover, we can say that $\displaystyle \lim_{t \to \infty}S(r,t)=\frac{A}{\mu_1}$ uniformly in any bounded subset of $[0,\infty)$.
\end{proof}

It states that if $\mathcal R_0<1$, then DFE is globally stable. Therefore, it satisfies one part of the definition of the reproductive number. We now need to verify the remaining part. In other words, we have to show that the DFE solution for $\mathcal R_0<1$ is unstable. This means that under certain conditions, the behavior of the equilibrium solution is significantly different. We are currently proving the remaining part of the definition, and it is expected to satisfy the condition. Moreover, we aim to identify a meaningful criterion or number that can distinguish the different behaviors or phenomena of each equilibrium solution, especially for unstable solutions.

\subsection{Persisting Case $\mathcal R_0>1$}
Then, to satisfy the remaining part of the definition, the remaining part would focus on the upper portion of the threshold. In this model, as the desired result could not be derived directly, we adopted a method analogous to that used in an earlier section, involving conditions on the diffusion coefficients and the summation of some compartments. First, we will prove a comparison lemma for the sum.
\begin{lemma} \label{comparison lem of sum}
If $d_1=d_2=d_3=d$, suppose that $T\in (0, \infty)$, $\bar h \in C^1([0,T])$, $\bar X, \bar Y \in C(\overline\Omega_T) \cap C^{2,1}(\Omega_T)$ and
\begin{equation*}
	\left \{ \begin{aligned}
		&\bar{X}_t-d\Delta \bar{X} \geq A - \max_{i=1,2,3} \{ \mu _i\} \bar{X}, &&(r,t)\in \Omega\\
		&\bar{Y}_t-d\Delta \bar{Y} -(\max_{i=1,2}\{r_i\}+\max_{i=2,3}\{\mu_i\})\bar{Y}, &&(r,t)\in \Omega_{\overline h(t)}\\
		&\bar{X}_r(0,t)\geq 0, \bar{Y}_r(0,t) \geq 0\\
		&\bar{Y}(r,t)=0, &&(r,t)\notin \Omega_{\overline h(t)}\\
		&\bar{h}'(t) \geq -\max\{\mu,\beta\}\bar{Y}_r(\bar{h}(t),t),\quad \bar{h}(0)>h_0,\\
		&\bar{X}(r,0) \geq S_0(r)+E_0(r)+I_0(r),\quad \bar{Y}_r(r,0)\geq E_0(r)+I_0(r).
	\end{aligned}
	\right.
\end{equation*}
Then the solution $(S,E,I;h)$ of the problem \cref{free boundary SEIS} with \cref{initial cond} satisfies $S(r,t)+E(r,t)+I(r,t)\leq\bar{X}(r,t)$, $E(r,t)+I(r,t)\leq\bar{Y}(r,t)$ for $r\in (0,\infty)$ and $t\in (0,T]$.
\end{lemma}

\begin{proposition}
	If $d_1=d_2=d_3=d$, and $\mathcal R_0^{\min}\left(:=\frac{\alpha A}{\max\limits_{i=1,2,3}\{\mu_i\}(\max\limits_{i=1,2}\{r_i\}+\max\limits_{i=2,3}\{\mu_i\})}\right)>1$, $h_0\leq \sqrt{\frac{d}{16k_0}}$ and $\mu$ and $\beta$ are less than equal to $\frac{3d}{4M}$, then $h_\infty < \infty$, where $k_0=\alpha C_1 -\max\limits_{i=1,2}\{r_i\}-\max\limits_{i=2,3}\{\mu_i\}$, $C=\max\{\lVert S_0 \rVert_\infty+\lVert E_0 \rVert_\infty+\lVert I_0 \rVert_\infty,\frac{A}{\max\limits_{i=1,2,3}\{\mu_i\}}\}$.
\end{proposition}
\begin{proof}
	Let
	\begin{equation*}
		\begin{aligned}
			&\bar{X}=C,\\
			&\bar{Y}=\left\{ \begin{aligned}
			&2Me^{-\nu t}(1-\frac{r^2}{\bar{h}^2(t)}), &&0\leq r \leq \bar h(t),\\
			&0, &&r>\bar h (t),
			\end{aligned} \right.
		\end{aligned}
	\end{equation*}
	and \begin{equation*}
		\bar h (t) =2h_0(2-e^{-\nu t}),\quad t\geq 0,
	\end{equation*}
	where $C=\max\{\lVert S_0 \rVert_\infty+\lVert E_0 \rVert_\infty+\lVert I_0 \rVert_\infty,\frac{A}{\max\limits_{i=1,2,3}\{\mu_i\}}\}$, $\nu$ and $M$ are positive constants to be chosen later. It is obvious that
	$\bar Y (r,t)=0$ and $\bar h(0)>h_0$
	by the construction of the function. If denoting
	\begin{equation*}
		k_0=\alpha C -\max\limits_{i=1,2}\{r_i\}-\max\limits_{i=2,3}\{\mu_i\},
	\end{equation*} since $R_0^{\min}>1$, we can say that $k_0>0$. Then we want to show the inequality to apply \cref{comparison lem of sum}. Then
	\begin{equation*}
			\bar X_t-d\Delta \bar X = 0 \geq A - \max_{i=1,2,3} \{ \mu _i\} \bar{X}
	\end{equation*}
	for $(r,t) \in \Omega_{\bar h}$ by definition of $C$ and $\bar X$. In addition, since
	\begin{equation*}
		\bar Y_t-d\Delta \bar Y = 2Me^{-\nu t}\left(\nu+\frac{r\nu}{\bar h(t)}+2\frac{r^2 \bar h'(t)}{\bar h ^3(t)}+2d\frac{r}{\bar h^2(t)}+2d\frac{n-1}{r\bar h^2(t)}\right),
	\end{equation*}
	$\bar Y$ satisfies the inequality
	\begin{equation*}
	\begin{aligned}
		&\bar Y_t-d\Delta \bar Y-(\alpha \bar X-\max\limits_{i=1,2}\{r_i\}-\max\limits_{i=2,3}\{\mu_i\})\bar Y \\
		&= 2Me^{-\nu t}\left(\nu+\frac{r\nu}{\bar h(t)}+2\frac{r^2 \bar h'(t)}{\bar h ^3(t)}+2d\frac{r}{\bar h^2(t)}+2d\frac{n-1}{r\bar h^2(t)}-k_0\left(1-\frac{r^2}{\bar h^2(t)}\right)\right)\\
		&\geq 2Me^{-\nu t}\left(-\nu -k_0 +\frac{d}{8h_0^2}\right),
	\end{aligned}
	\end{equation*}
	for $(r,t) \in \Omega_{hT}$. Now, we will set $\nu =\frac{d}{16h_0^2}$. On the other hand, we can say that
	\begin{equation*}
	\begin{aligned}
		&\bar h'(t)=2h_0\nu e^{-\nu t}, \quad \text{and}\\
		&-\max\{\nu,\beta\}\bar Y_r(\bar h(t),t)=2M\max \{\nu,\beta\}\frac{e^{-\nu t}}{\bar h(t)}.
	\end{aligned}
	\end{equation*}
	Furthermore, it is clear that $\bar{X}(r,0)\geq X_0(r)=C\geq \lVert S_0 \rVert_\infty + \lVert E_0 \rVert_\infty+\lVert I_0 \rVert_\infty$ holds by the construction of constant function $\bar X$. Similarly,
	\begin{equation*}
		\bar Y (r,0)=2M\left(1-\frac{r^2}{4h_0^2}\right)\geq \frac{3}{2}M,
	\end{equation*}
	for $0\leq r \leq h_0$. Thus, if we choose $M=\frac{2}{3}(\lVert E_0 \rVert_\infty + \lVert I_0 \rVert_\infty)$, since $\bar h(t) \leq 4h_0$ for $t\geq 0$,
	\begin{equation*}
		\bar Y (r,0) \geq \lVert E_0 \rVert_\infty + \lVert I_0\rVert _\infty.
	\end{equation*}
	Moreover, we want to show
	\begin{equation*}
		\bar h'(t) \geq -\max \{ \mu, \beta \} \bar Y_r(\bar h(t),t),
	\end{equation*}
	and it is equal to show that
	\begin{equation*}
		2h_0\nu e^{-\nu t} \geq \frac{4}{3}M\max\{\mu,\beta\}\frac{e^{-\nu t}}{2\bar h(t)}.
	\end{equation*}
	Since $\bar h(t) \leq 4 h_0$, it is also same with
	\begin{equation*}
		h_0 \nu \geq \frac{M}{3} \max \{ \mu,\beta \} \frac{1}{4 h_0}.
	\end{equation*}
	Moreover, $\mu$ and $\beta$ have to satisfy
	\begin{equation*}
		\frac{3d}{4M}\geq \max \{ \mu, \beta \},
	\end{equation*}
	by definition of $\nu$.
	Thus, if $\mu, \beta \leq \frac{3d}{4M}$,
	\begin{equation*}
		\bar h'(t) \geq -\max\{\mu,\beta\}\bar{Y}_r(\bar{h}(t),t).
	\end{equation*}
	In addition to, we have
	\begin{equation*}
		\frac{d}{16 h_0^2}\geq 0, \quad \text{i.e.,} \quad h_0 \leq \sqrt{\frac{d}{16 k_0}}.
	\end{equation*}
	Finally, we can apply \cref{comparison lem of sum} for $\bar X$ and $\bar Y$, and it is also true that
	\begin{equation*}
		\begin{aligned}
			&S(r,t)+E(r,t)+I(r,t) \leq \bar X (r,t)\\
			&h(t) \leq \bar h(t)\\
			&E(r,t)+I(r,t) \leq \bar Y(r,t).
		\end{aligned}
	\end{equation*}
	Furthermore, by non-negativity of each compartment, $S(r,t)\leq \bar X (r,t)$. Therefore, we have
	\begin{equation*}
		h_\infty \leq \lim_{t\to \infty}\bar h (t) =4h_0 < \infty.
	\end{equation*}
\end{proof}

\begin{lemma} \label{lemma 4.5}
If $h_\infty < \infty$, then
\begin{equation*}
	\lim_{t \to \infty}\lVert E(\cdot ,t)\rVert_{C([0,h(t)])}=\lim_{t \to \infty}\lVert I(\cdot ,t)\rVert_{C([0,h(t)])}=0.
\end{equation*}
\end{lemma}
\begin{proof}
	Suppose that
	\begin{equation*}
		\lim_{t\to \infty}\lVert (E+I)(\cdot,t)\rVert_{C([0,h(t)])}\neq 0. 
	\end{equation*}
	Then, 
	\begin{equation*}
		\limsup_{t \to \infty}\lVert (E + I) (\cdot , t)\rVert_{C([0,h(t)])}=\delta>0,
	\end{equation*}
	 and there exists $(r_k,t_k)\in \bar \Omega_h$ such that
	\begin{equation*}
	\begin{gathered}
		 (E+I)(r_k,t_k)\geq \frac{\delta}{2} \quad\text{ for }\forall k\in \mathbb{N},\\
		\lim_{k \to \infty}t_k= \infty.
	\end{gathered}
	\end{equation*}
	Since $0 \leq r_k < h(t) < h_\infty < \infty$, $\{r_n\}$ has a subsequence which converges to $r_0\in [0,h_\infty)$. Thus, with out loss of generality, we can assume that $\displaystyle\lim_{k\to \infty}r_k=r_0$. Define
	\begin{equation*}
		\begin{aligned}
			S_k(r,t)=S(r,t_k+t)\\
			E_k(r,t)=E(r,t_k+t)\\
			I_k(r,t)=I(r,t_k+t)
		\end{aligned}
	\end{equation*}
	for $(r,t) \in (0,h(t_k+t))\times (-t_k+\infty)$. Then, $\{(S_k,E_k,I_k)\}$ has a subsequence $\{(S_{k_i},E_{k_i},I_{k_i})\}$ such that $\displaystyle \lim_{i \to \infty}(S_{k_i},E_{k_i},I_{k_i})=(\tilde S,\tilde E,\tilde I)$ and it satisfies
	\begin{equation*}
		\left \{ \begin{aligned}
			&\tilde S_t-d\Delta \tilde S = A - \alpha \tilde I \tilde S -\mu_1 \tilde S+r_2\tilde I+r_1\tilde E,\\
			&\tilde E_t-d\Delta \tilde E=(1-p)\alpha \tilde I \tilde S -(\beta_1+r_1+\mu_2)\tilde E,\\
			&\tilde I_t-d\Delta \tilde I =p\alpha \tilde I \tilde S +\beta_1 E-(r_2+\mu_3)\tilde I
		\end{aligned} \right.
	\end{equation*}
	for $(r,t)\in (0,h_\infty)\times (-\infty,\infty)$. Since $(\tilde E+\tilde I)(r_0,0)\geq \frac{\delta}{2}$, $(\tilde E +\tilde I)>0$ in $[0,h_\infty) \times (-\infty,\infty)$.
	By using
	\begin{equation*}
	\begin{aligned}
		(\tilde E +\tilde I)_t-d\Delta (\tilde E +\tilde I)&=\alpha \tilde I \tilde S -(r_1+\mu_2)\tilde E -(r_2+\mu_3) \tilde I\\
		&\geq -(r_1+\mu_2)\tilde E -(r_2+\mu_3) \tilde I\\
		&\geq -\max\{r_1+\mu_2,r_2+\mu_3\}(\tilde E+ \tilde I),
	\end{aligned}
	\end{equation*}
	Hopf lemma implies $(\tilde E + \tilde I)_r(h_\infty, 0)\leq -\sigma_0$ for some $\sigma_0>0$. By using straightening lemma \cref{straightening lem}, let
	\begin{equation}
		\begin{aligned}
			&s=\frac{h_0r}{h(t)},\\
			&u(s,t)=S(r,t),\\
			&v(s,t)=E(r,t),\\
			&w(s,t)=I(r,t).
		\end{aligned}
	\end{equation}
	Then,
	\begin{equation*}
	\begin{aligned}
		&s_t=-s\frac{h'(t)}{h(t)},\\
		&E_t=v_t+v_s\left(-s\frac{h'(t)}{h(t)}\right), \quad E_r=v_s\frac{h_0}{h(t)},\quad \Delta E=\Delta v \left(\frac{h_0}{h(t)}\right)^2,\\
		&I_t=w_t+w_s\left(-s\frac{h'(t)}{h(t)}\right),\quad I_r=w_s\frac{h_0}{h(t)},\quad \Delta I = \Delta w \left(\frac{h_0}{h(t)}\right)^2,
	\end{aligned}
	\end{equation*}
	and $v$ and $w$ satisfy
	\begin{equation*}
		\left \{ \begin{aligned}
			&v_t -sv_s\frac{h'(t)}{h(t)}-d\Delta v (\frac{h_0}{h(t)})^2=(1-p)\alpha uw-(\beta_1+r_1+\mu_2)v, \quad &&(r,t)\in\Omega_{h_0},\\
			&w_t-sw_s\frac{h'(t)}{h(t)}-d\Delta w(\frac{h_0}{h(t)})^2=p\alpha uw+\beta_1 v-(r_2+\mu_3)w, \quad &&(r,t)\in\Omega_{h_0},\\
			&v_s(0,t)=v(h_0,t)=w_s(0,t)=w(h_0,t)=0, \quad &&t>0,\\
			&v(s,0)=E_0(s), w(s,0)=I_0(s), \quad &&0\leq s \leq h_0.
		\end{aligned} \right.
	\end{equation*}
	By uniform estimate \cref{uniform SEI} and \cref{uniform h'},
	\begin{equation*}
	\begin{aligned}
		&\lVert (1-p)\alpha uw-(\beta_1+r_1+\mu_2)v\rvert_{L^\infty}\leq C_2,\\
		&\lVert p\alpha uw-\beta_1 v (r_2+\mu_3)w\rVert_{L^\infty}\leq C_2,\\
		&\lVert \frac{h'(t)}{h(t)}s\rVert_{L^\infty}\leq C_3.
	\end{aligned}
	\end{equation*}
	By $L^p$ theory and the Sobolev embedding for $W^{2,p}$,
	\begin{equation*}
		\begin{aligned}
			\lVert v \rVert_{C^{1+\alpha}(\overline{\Omega_{h_0}})}\leq C',\\
			\lVert w \rVert_{C^{1+\alpha}(\overline{\Omega_{h_0}})}\leq C',
		\end{aligned}
	\end{equation*}
	where $C'$ is a constant depending only on $\alpha, h_0, C_1, C_2, \lVert E_0\rVert_{C^2},$ and $\lVert I_0\rVert_{C^2}$.
	Therefore, for any $0<\alpha<1$, there exists a constant $\tilde C$ which only depends on $\alpha, h_0, \lVert E_0+I_0 \rVert_{C^{1+\alpha}}([0,h_0])$ and $h_\infty$, such that
	\begin{equation*}
		\lVert E+I \rVert_{C^{1+\alpha}}(\overline{\Omega_{h_0}})+ \lVert h \rVert_{C^{1+\frac{\alpha}{2}}([0,\infty))}\leq \tilde C.
	\end{equation*}
	Since $\lVert h \rVert_{C^{1+\frac{\alpha}{2}}([0,\infty))}\leq \tilde C$ and $h(t)$ is bounded,
	\begin{equation*}
		\lim_{t \to \infty}h'(t)=0,
	\end{equation*}
	so
	\begin{equation*}
		\lim_{t_k \to \infty}I_r(h(t_k),t_k) = 0,
	\end{equation*}
	by the free boundary condition.
	Moreover, by the inequality of $E$ and $I$, 
	\begin{equation*}
		\lim_{k \to \infty}(E+I)_r(h(t_k),t_k)=\lim_{k \to \infty}((E+I)_k)_r(h(t_k),0)= (\tilde E+\tilde I)_r(h_\infty,0)
	\end{equation*}
	and it contradicts the fact
	\begin{equation*}
		(\tilde E +\tilde I)_r{h_\infty},0)\leq - \sigma_0.
	\end{equation*}
	Therefore,
	\begin{equation*}
	\lim_{t \to \infty}\lVert E(\cdot ,t)\rVert_{C([0,h(t)])}=\lim_{t \to \infty}\lVert I(\cdot ,t)\rVert_{C([0,h(t)])}=0.
	\end{equation*}
	\end{proof}

\begin{proposition}
	If $\mathcal{R}_0^{\min}\left(:=\frac{\alpha A}{\max\limits_{i=1,2,3}\{\mu_i\}(\max\limits_{i=1,2}\{r_i\}+\max\limits_{i=2,3}\{\mu_i\})}\right)>1$ and $h_0>h_0^*$, then $h_\infty = \infty$.
\end{proposition}
\begin{proof}
	Assume that $h_\infty < \infty$. By \cref{lemma 4.5},
	\begin{equation*}
		\lim_{t \to \infty}\lVert E(\dot,t)\rVert_{C([0,h(t)])}=\lim_{t \to \infty}\lVert I(\dot,t)\rVert_{C([0,h(t)])}=0.
	\end{equation*}
	Thus, since $\displaystyle \lim_{t\to \infty}S(r,t)=\frac{A}{\mu_1}$ uniformly for $r \in [0,h_0]$. Therefore, there exists $T^*>0$ for $\epsilon>0$ such that $\displaystyle S(r,t) \geq \frac{A}{\mu_1}-\epsilon$ for $r\in [0,h_(t)), t>T^*$. Then $E(r,t)$ and $I(r,t)$ satisfy
	\begin{equation*}
		\left \{ \begin{aligned}
			&E_t-d\Delta E \geq (1-p)\alpha I \frac{A}{\mu_1}-\beta_1 E-r_1 E-\mu_2 E, \quad && 0<r<h_0, t>T^*,\\
			&I_t-d\Delta I \geq p \alpha I \frac{A}{\mu_1}+\beta_1 E-r_2 I-\mu_3, &&0<r<h_0, t>T^*,\\
			&E_r(0,t)=I_r(0,t)=0, E(h_0,t), I(h_0,t)\geq0 &&t>T^*,\\
			&E(r,T^*), I(r,T^*)>0, &&0\leq r <h_0.
		\end{aligned} \right.
	\end{equation*}
	Since $\displaystyle \lim_{t \to \infty} E=0$,
	\begin{equation*}
		\left \{ \begin{aligned}
			&I_t-d\Delta I \geq p \alpha I \frac{A}{\mu_1}-(r_2+\mu_3)I=(p\alpha \frac{A}{\mu_1}-r_2-\mu_3)I, &&0<r<h_0, t>T^*,\\
			&I_r(0,t)=0, \quad I(h_0,t)\geq 0, &&t>T^*,\\
			&I(r,T^*)>0, &&0\leq r < h_0.
		\end{aligned} \right.
	\end{equation*}
	Thus, we can construct the subsolution $\underline I$ of
	\begin{equation*}
		\left \{ \begin{aligned}
			&\underline I_t-d\Delta \underline I = (p\alpha \frac{A}{\mu_1}-r_2-\mu_3)\underline I, &&0<r<h_0, t>T^*,\\
			&\underline I_r(0,t)=0, \quad \underline I(h_0,t)\geq 0, &&t>T^*,\\
			&\underline I(r,T^*)=I(r,T^*), &&0= r < h_0.
		\end{aligned} \right.
	\end{equation*}
	By assumption, since $h_0>h_0^*$, there is sufficiently small $\epsilon>0$ such that
	\begin{equation*}
		p\alpha(\frac{A}{\mu_1}-r_2-\mu_3)>d \lambda_1(h_0),
	\end{equation*}
	where $\lambda_1(R)$ is the principal eigenvalue of the operator $-\Delta$ in open ball $B_R$ with radius $R$. Then $\underline{I}$ is unbounded in $(0,h_0)\times [T^*,\infty)$ and it makes contradiction. Therefore $\displaystyle \lim_{t \to \infty} \lVert I \rVert \neq 0$ and $h_\infty = \infty$.
\end{proof}

In conclusion, this result clearly defines the reproductive number well.

\begin{theorem}
	\begin{equation*}
		\mathcal{R}_0=\left\{ \mathcal{R}_0^{\min}\left(:=\frac{\alpha A}{\max\limits_{i=1,2,3}\{\mu_i\}(\max\limits_{i=1,2}\{r_i\}+\max\limits_{i=2,3}\{\mu_i\})}\right), \mathcal{R}_0^{\max}\left(:=\frac{\alpha A}{\min\limits_{i=1,2,3}\{\mu_i \}\min \{ r_1+\mu_2,\ r_2+\mu_3\}}\right)\right\}
		\end{equation*}
			is the reproductive number. In other words, if $\mathcal R_0^{\max}<1$, then the DFE is locally stable, and if $\mathcal R_0^{\min}>1$, then the DFE is unstable. Moreover, for $\mu_1<\mu_2<\mu_3$,	and $r_1+\mu_2<r_2+\mu_3$,
	\begin{equation*}
	\mathcal R_0=\left\{\mathcal{R}_0^{\min}\left(:=\frac{\alpha A}{\mu_3(r_1+\mu_3)}\right), \mathcal{R}_0^{\max}\left(:=\frac{\alpha A}{\mu_1(r_2+\mu_3)}\right)\right\}.
	\end{equation*}
	In other words, if $\displaystyle \mathcal{R}_0^{\min}=\frac{\alpha A}{\mu_3(r_1+\mu_3)}<1,$
			then the DFE is locally stable, and if
				$\displaystyle \mathcal{R}_0^{\max}=\frac{\alpha A}{\mu_1(r_2+\mu_3)}<1$,
			then the DFE is unstable.
\end{theorem}

\subsection{The Eigen Value Problems}
The nonlinear eigenvalue problem can describe the asymptotic profiles of a corresponding parabolic flow\cite{Lee_2008}. Thus, we consider the eigenvalue problem to identify the speed of convergence of the infected and exposed populations.

We consider the solution $E(r,t)$ and $I(r,t)$  of the equations \cref{free boundary SEIS} and 
first straighten the free boundary to make the boundary fixed. 

\begin{lemma} \label{straightening lem2}
	Suppose that $\displaystyle \lim _{t \to \infty}h(t)=h_\infty$ for some $h_\infty >0$. The free boundary $h(t)$ of the system of equations \cref{free boundary SEIS} with \cref{initial cond} is diffeomorphic to $h_\infty$ for $d_1=d_2=d_3=d$.
\end{lemma}

\begin{proof}
	Construct $\xi(s)$ be a function in $C^3([0, \infty))$ satisfying
		\begin{equation*}
		\left \{ \begin{aligned}
			&\xi(s)=1, &&\text{\ if \ }\lvert s-h_\infty \rvert < \frac{h_\infty}{8},
			\\
			&\xi(s)=0, &&\text{\ if \ }\lvert s-h_\infty \rvert > \frac{h_\infty}{2},
			\\
			&\lvert \xi '(s) \rvert < \frac{5}{h_\infty}, &&\text{\ for all\ } s.
		\end{aligned} \right.
		\end{equation*}
		Consider the transformation
		\begin{equation*}
			(y,t) \mapsto (x,t), \text{ where } x=y+\xi(\lvert y \rvert )\left(h(t)-h_\infty \frac{y}{\lvert y \rvert}\right), \ y \in \mathbb{R}^n,
		\end{equation*}
		which leads to the transformation
		\begin{equation*}
			(s,t) \mapsto (r,t)\text{ with } r=s+\xi(s)(h(t)-h_\infty),\ 0 \leq s < \infty.
		\end{equation*}
		As long as
		\begin{equation*}
			\lvert h(t)-h_\infty \rvert \leq \frac{h_\infty}{8},
		\end{equation*}
		the above transformation $x \rightarrow y$ is a diffeomorphism from $\mathbb{R}^n$ onto $\mathbb{R}^n$ and the transformation $s \rightarrow r$ is also a diffeomorphism from $[0,+\infty)$ onto $[0,+\infty)$. So $s$ is radially symmetric in the same way that $r$ is and we can define $\Delta_s w = w_{ss}+ \frac{n-1}{s}w_s$. Moreover, it changes the free boundary $r=h(t)$ to fixed boundary $s=h_\infty$.
\end{proof}

\begin{lemma} \label{equation changing}
Suppose that $\displaystyle \lim _{t \to \infty}h(t)=h_\infty$ for some $h_\infty >0$. If $d_1=d_2=d_3=d$ for some $d>0$, the solution of the equations \cref{free boundary SEIS} of the Exposed and the Infected individuals are equal to the solution of the equations,
\begin{equation} \label{modified free boundary SEIS4}
	\left \{ \begin{aligned}
		u_t-Xd\Delta_s u-(Yd+Zh'(t)+Wd)u_s=(1-p)\alpha vw-(\beta_1+r_1+\mu_2)u,\\
		v_t-Xd\Delta_s u-(Yd+Zh'(t)+Wd)u_s=p\alpha vw+\beta_1 u-(r_2+\mu_3)v,
	\end{aligned} \right.
\end{equation}
for $0<s<h_\infty$ and $t>0$, and
\begin{equation} \label{modified free boundary SEIS4 cond}
	\left \{ \begin{aligned}
		&h'(t)=-\beta v_s(h_0,t)-\mu w_s(h_0,t), \quad &&s=h_\infty,\\
		&u(s,t)=v(s,t)=0, &&s\geq h_\infty,
	\end{aligned} \right.
\end{equation}
where
\begin{equation*}
	\begin{aligned}
		&\sqrt{X(h(t),s)}\coloneqq \frac{\partial s}{\partial r}=\frac{1}{1+\xi'(s)(h(t)-h_0)},
			\\
		&Y(h(t),s)\coloneqq \frac{\partial^2s}{\partial r^2}=-\frac{\xi''(s)(h(t)-h_0)}{[1+\xi'(s)(h(t)-h_0)]^3},
			\\
		&Z(h(t),s)\coloneqq -\frac{1}{h'(t)}\frac{\partial s}{\partial t}=\frac{\xi(s)}{1+\xi'(s)(h(t)-h_0)},
			\\
		&W(h(t),s)\coloneqq(n-1)X\frac{(s\xi'-\xi)(h(t)-h_0)}{s(s+\xi(s)(h(t)-h_0))}.\,
	\end{aligned}
\end{equation*}
\end{lemma}

\begin{proof}
By \cref{straightening lem2}, free boundary $h(t)$ is changed to $h_\infty$. Now, set 
\begin{equation*}
	\begin{aligned}
	E(r,t) = E(s+\xi(s)(h(t)-h_\infty),t)=:u(s,t),\\
	I(r,t) = I(s+\xi(s)(h(t)-h_\infty),t)=:v(s,t).
	\end{aligned}
\end{equation*}
If we differentiate both sides with respect to $s$,
\begin{equation*}
	\begin{aligned}
		(1+\xi'(s)(h(t)-h_\infty))E_r=u_s,\\
		(1+\xi'(s)(h(t)-h_\infty))I_r=v_s,
	\end{aligned}
\end{equation*}
so we can obtain
\begin{equation*}
	\begin{aligned}
		E_r = \frac{u_s}{1+\xi'(s)(h(t)-h_\infty},\\
		I_r = \frac{v_s}{1+\xi'(s)(h(t)-h_\infty}.
	\end{aligned}
\end{equation*}
And if we differentiate both sides twice with respect to $s$,
\begin{equation*}
	\begin{aligned}
		&\xi''(s)(h(t)-h_\infty)E_r+(1+\xi'(s)(h(t)-h_\infty)\frac{\partial}{\partial s}E_r=u_{ss},\\
		&\xi''(s)(h(t)-h_\infty)E_r+(1+\xi'(s)(h(t)-h_\infty)\frac{\partial r}{\partial s}\frac{\partial}{\partial r}E_r=u_{ss},\\
		&\xi'(s)(h(t)-h_\infty)E_r+(1+\xi'(s)(h(t)-h_\infty))^2E_{rr}=u_{ss}.
	\end{aligned}
\end{equation*}
Similarly,
\begin{equation*}
	\xi''(s)((h(t)-h_\infty)I_r+(+\xi'(s)(h(t)-h_\infty))^2I_{rr}=v_{ss},
\end{equation*}
so we can obtain
\begin{equation*}
	\begin{aligned}
		E_{rr}=\frac{u_{ss}-\xi''(s)(h(t)-h_\infty)E_r}{(1+\xi'(s)(h(t)-h_\infty))^2}=\frac{(1+\xi'(s)(h(t)-h_\infty))u_{ss}-\xi''(s)(h(t)-h_\infty)u_s}{(1+\xi'(s)(h(t)-h_\infty))^3},
	\end{aligned}
\end{equation*}
and 
\begin{equation*}
	I_{rr}= \frac{v_{ss}-\xi''(s)(h(t)-h_\infty)I_r}{(1+\xi'(s)(h(t)-h_\infty))^2}=\frac{((1+\xi'(s))(h(t)-h_\infty))v_{ss}-\xi''(s)(h(t)-h_\infty)u_s}{(1+\xi'(s)(h(t)-h_\infty))^3}.
\end{equation*}
In other case, if we differentiate both sides with respect to $t$,
\begin{equation*}
	\begin{aligned}
		&\xi(s)h'(t)E_r+E_t=u_t,\\
		&\frac{\xi(s)h'(t)}{1+\xi'(s)(h(t)-h_\infty)}u_s+E_t=u_t,
	\end{aligned}
\end{equation*}
and
\begin{equation*}
	\frac{\xi(s)h'(t)}{1+\xi'(s)(h(t)-h_\infty)}v_s+I_t=v_t,
\end{equation*}
so we can obtain
\begin{equation*}
	\begin{aligned}
		E_t=u_t-\frac{\xi(s)h'(t)}{1+\xi'(s)(h(t)-h_\infty)}u_s,\\
		I_t=v_t-\frac{\xi(s)h'(t)}{1+\xi'(s)(h(t)-h_\infty)}v_s.
	\end{aligned}
\end{equation*}
By using these calculations, we can change of the eqution\cref{free boundary SEIS} of the Exposed $E$,
\begin{equation}\label{modified free boundary SEIS3}
	\begin{aligned}
		E_t-d\Delta_rE&=u_t-\frac{\xi(s)h'(t)}{1+\xi'(s)(h(t)-h_\infty)}u_s\\
		&\quad -d\left(\frac{(1+\xi'(s)(h(t)-h_\infty))u_{ss}-\xi''(s)(h(t)-h_\infty)u_s}{(1+\xi'(s)(h(t)-h_\infty))^3}\right.\\
		&\quad \left.+\frac{n-1}{r}\frac{u_s}{1+\xi'(s)(h(t)-h_\infty)}\right)\\
		&=u_t-\frac{\xi(s)h'(t)}{1+\xi'(s)(h(t)-h_\infty)}u_s-\frac{du_{ss}}{(1+\xi'(s)(h(t)-h_\infty))^2}\\
		&\quad +\frac{d\xi''(s)(h(t)-h_\infty)u_s}{(x+\xi'(s)(h(t)-h_\infty))^3}-d\frac{n-1}{r}\frac{u_s}{1+\xi'(s)(h(t)-h_\infty)}\\
		&=u_t-\frac{\xi(s)h'(t)}{1+\xi'(s)(h(t)-h_\infty)}u_s-\frac{d\Delta_s u}{(1+\xi'(s)(h(t)-h_\infty))^2}\\
		&\quad +\frac{d\xi''(s)(h(t)-h_\infty)u_s}{(1+\xi'(s)(h(t)-h_\infty))^3}\\
		&\quad +\frac{d(n-1)u_s(-s\xi'(s)+\xi(s))(h(t)-h_\infty)}{s(s+\xi(s)(h(t)-h_\infty)(1+\xi'(s)(h(t)-h_\infty))^2}.
	\end{aligned}
\end{equation}
If we define
\begin{equation*}
	\begin{aligned}
		&\sqrt{X(h(t),s)}\coloneqq \frac{\partial s}{\partial r}=\frac{1}{1+\xi'(s)(h(t)-h_0)},
			\\
		&Y(h(t),s)\coloneqq \frac{\partial^2s}{\partial r^2}=-\frac{\xi''(s)(h(t)-h_0)}{[1+\xi'(s)(h(t)-h_0)]^3},
			\\
		&Z(h(t),s)\coloneqq -\frac{1}{h'(t)}\frac{\partial s}{\partial t}=\frac{\xi(s)}{1+\xi'(s)(h(t)-h_0)},
			\\
		&W(h(t),s)\coloneqq(n-1)X\frac{(s\xi'-\xi)(h(t)-h_0)}{s(s+\xi(s)(h(t)-h_0))}.\,
	\end{aligned}
\end{equation*}
then we can rewrite the equation\cref{modified free boundary SEIS3}
\begin{equation*}
	u_t-Xd\Delta_s u -(Yd+Zh'(t)+Wd)u_s=(1-p)\alpha vw-(\beta_1+r_1+\mu_2)u.
\end{equation*}
Similarly, we had a direct calculation for $I$ and it is changed to
\begin{equation*}
	\begin{aligned}
		I_t-d_3\Delta_r I=&v_t-\frac{\xi(s)h'(t)}{1+\xi'(s)(h(t)-h_\infty)}v_s\\
		&-d\left( \frac{(1+\xi'(s)(h(t)-h_\infty))v_{ss}-\xi''(s)(h(t)-h_\infty)v_s}{(1+\xi'(s)(h(t)-h_\infty))^3} \right. \\
		&\left. \frac{n-1}{r}\frac{v_s}{1+\xi'(s)(h(t)-h_\infty)} \right),
	\end{aligned}
\end{equation*}
so
\begin{equation*}
	I_t-d\Delta_r I = v_t-Xd\Delta_s v-(Yd+Zh'(t)+Wd)v_s=p\alpha vw+\beta_1u-(r_2+\mu_3)v.
\end{equation*}
Therefore, the equations become
\begin{equation*}
	\left \{ \begin{aligned}
		&u_t-Xd\Delta_s u-(Yd+Zh'(t)+Wd)u_s=(1-p)\alpha vw-(\beta_1+r_1+\mu_2)u,\\
		&v_t-Xd\Delta_s v-(Yd+Zh'(t)+Wd)v_s=p\alpha vw+\beta_1 u-(r_2+\mu_3)v,
	\end{aligned} \right.
\end{equation*}
for $(s,t)\in \Omega_{h_\infty}$, and
\begin{equation*}
	\left \{ \begin{aligned}
		&h'(t)=-\beta v_s(h_0,t)-\mu w_s(h_0,t), \quad &&s=h_\infty,\\
		&u(s,t)=v(s,t)=0, &&s\geq h_\infty.
	\end{aligned} \right.
\end{equation*}
\end{proof}
By using a diffeomorphism, we changed the equation with a free boundary into the equation with a fixed boundary. We now consider the case in which DFE is stable. If $\mathcal{R}_0<1$, $S(r,t)\rightarrow\frac{A}{\mu_1}$ as $t\to \infty$ by the \cref{theorem 4.2} and $X(t,s)\to 1, Y(t,s)\to 0, Z(t,s)\to \xi (s), W(t,s) \to 0$ and $h'(t)\to 0$as $t \to \infty$, so the equations \cref{modified free boundary SEIS4} becomes
\begin{equation} \label{DFE ODE}
	\left \{ \begin{aligned}
		&d\Delta \phi -(\beta_1+r_1+\mu_2)\phi+(1-p)\alpha \frac{A}{\mu_1}\psi=0, &&\text{in } \Omega_{h_\infty}\\
		&d\Delta\psi +(p\alpha \frac{A}{\mu_1}-r_2-\mu_3)\psi+\beta_1\phi=0,  &&\text{in } \Omega_{h_\infty}\\
		&\phi=0, &&\text{on } \partial \Omega_{h_\infty},\\
		&\psi=0, &&\text{on } \partial \Omega_{h_\infty},
	\end{aligned} \right.
\end{equation}
on DFE. Thus, let
\begin{equation} \label{operator L}
	\mathcal{L}'=\begin{pmatrix}
		d\Delta \phi +a_{11} & a_{12} \\
		a_{22} & d\Delta\psi+ a_{21}
	\end{pmatrix}
\end{equation}
where
\begin{equation*}
\begin{pmatrix}
a_{11} & a_{12} \\
a_{21} & a_{22}
\end{pmatrix}
=
\begin{pmatrix}
-(\beta_1 + r_1 + \mu_2) & (1 - p)\alpha \frac{A}{\mu_1} \\
p \alpha \frac{A}{\mu_1} - (r_2 + \mu_3) & \beta_1
\end{pmatrix}.
\end{equation*}
Then we can say that
\begin{equation*}
	\begin{pmatrix}
		\phi \\
		\psi
	\end{pmatrix}
	\text{satisfies }
	\mathcal{L}'
	\begin{pmatrix}
		\phi \\
		\psi
	\end{pmatrix}
	=0,
\end{equation*}
on DFE. To analyze the solution of the equation\cref{DFE ODE}, we adopted the method introduced by \cite{Lee_2008}. It is well known that the Laplace operator has a countable discrete set of eigenvalues $\Sigma = \{\lambda_i \vert \lambda_1< \lambda_2 < \cdots < \lambda_n < \cdots \}$, whose eigen-functions $\{\phi_n\}$ span $W_0^{1,2}(\Omega)$, where $\phi_n$ is a normalized eigen-function corresponding to $\lambda_n$. Then, $u_n(x,t)=e^{-\lambda_n^t}\phi_n(x)$ is the solution of the heat equation with initial data $\phi_n(x)$.

\begin{proposition}
If we define
\begin{equation} \label{energy functional}
	\mathcal{E}[\phi,\psi]=\int_{\Omega} da_{22}\lvert \nabla \phi \rvert^2+da_{12} \lvert \nabla \psi \rvert^2 +a_{11}a_{22} \phi^2+2a_{12}a_{22}\phi\psi+a_{12}a_{21} \psi^2,
\end{equation}
then $\mathcal{E}$ corresponds to linear elliptic operator
\begin{equation*}
	\mathcal{L} \coloneqq \begin{pmatrix}
	a_{22}(d\Delta \phi +a_{11}) & a_{12}a_{22} \\
		a_{12}a_{22} & a_{12}(d\Delta\psi+ a_{21})
\end{pmatrix}.
\end{equation*}
\end{proposition}

\begin{proof}
If we define
\begin{equation*}
	\mathcal{E}[\phi,\psi]=\int_{\Omega} da_{22}\lvert \nabla \phi \rvert^2+da_{12} \lvert \nabla \psi \rvert^2 +a_{11}a_{22} \phi^2+2a_{12}a_{22}\phi\psi+a_{12}a_{21} \psi^2,
\end{equation*}
then,
\begin{equation*}
	\begin{aligned}
\mathcal{E}[\phi + \epsilon \varphi, \psi + \epsilon \chi] 
&= \int_{\Omega} a_{22} d |\nabla(\phi + \epsilon \varphi)|^2 
+ a_{12} d |\nabla(\psi + \epsilon \chi)|^2 \\
&\quad + a_{11} a_{22} (\phi + \epsilon \varphi)^2 + 2 a_{12} a_{22} (\phi + \epsilon \varphi)(\psi + \epsilon \chi) + a_{12} a_{21} (\psi + \epsilon \chi)^2 \\
&= \int_{\Omega} a_{22} d |\nabla \phi + \epsilon \nabla \varphi|^2 + a_{12} d |\nabla \psi + \epsilon \nabla \chi|^2 \\
&\quad + a_{11} a_{22} (\phi^2 + 2\epsilon \phi \varphi + \epsilon^2 \varphi^2) + a_{12} a_{21} (\psi^2 + 2\epsilon \psi \chi + \epsilon^2 \chi^2) \\
&\quad +2a_{12} a_{22}(\phi \psi + \epsilon \varphi \psi + \epsilon \phi \chi + \epsilon^2\varphi \chi),
\end{aligned}
\end{equation*}
so if we calculate derivative by $\epsilon$,
\begin{equation*}
	\begin{aligned}
\left. \frac{d}{d\epsilon} \mathcal{E}[\phi + \epsilon \varphi, \psi + \epsilon \chi] \right|_{\epsilon=0}
c+ 2a_{12}d \lvert \nabla \psi \cdot \nabla \chi \rvert \\
&\quad + 2a_{11}a_{22}\phi\varphi + 2a_{12}a_{22}(\psi\varphi + \phi\chi) + 2a_{12}a_{21}\psi\chi \\
&= \int_{\Omega} (-2a_{22}d (\Delta \varphi) \phi + 2a_{11}a_{22} \varphi \phi + 2a_{12}a_{22}\chi \phi ) \\
&\quad + (-2a_{12}d (\Delta \chi) \psi  + 2a_{12}a_{21} \chi \psi + 2a_{12}a_{22}\varphi \psi ) \\
&= \int_{\Omega} 
\begin{pmatrix}
-2a_{22}d_2 \Delta \phi + 2a_{11}a_{22}\phi + 2a_{12}a_{22}\psi \\
-2a_{12}d_3 \Delta \psi + 2a_{12}a_{21}\psi + 2a_{12}a_{22}\phi
\end{pmatrix}
\cdot
\begin{pmatrix}
\varphi \\ \chi
\end{pmatrix} \\
&= 2 \mathcal{L}
\begin{pmatrix}
\varphi \\ \chi
\end{pmatrix}
\cdot
\begin{pmatrix}
\phi \\ \psi
\end{pmatrix}
= 0
\end{aligned}
\end{equation*}
where
\begin{equation*}
	\mathcal{L} \coloneqq \begin{pmatrix}
	a_{22}(d\Delta \phi +a_{11}) & a_{12}a_{22} \\
		a_{12}a_{22} & a_{12}(d\Delta\psi+ a_{21})
\end{pmatrix}.
\end{equation*}
\end{proof}

By using the energy functional and linear operator, we will solve the optimization and find the solution of the elliptic eigenvalue problem. Then the eigenfunction and eigenvalue would be the key to the speed of the convergence.
\begin{lemma}
	The energy functional $\mathcal{E}$\cref{energy functional} is coercive.
\end{lemma}

\begin{proof}
	Since $a_{12},a_{22},d_2,d_3>0$,
	\begin{equation*}
		\int_{\Omega_h}a_{22}d_2\lvert \nabla \phi \rvert^2 + a_{12}d_3\lvert \nabla \psi \rvert^2 \geq \min \{ a_{22}d_2, a_{12}d_3 \} \int_{\Omega_h}\vert \nabla \phi \rvert^2 + \lvert \nabla \psi \rvert^2.
	\end{equation*}
	Moreover,
	\begin{equation*}
		\int_{\Omega_h}a_{11}a_{22}\phi^2+a_{12}a_{21}\psi^2 \geq -\max \{ \lvert a_{11} \rvert a_{22}, a_{12}\lvert a_{21} \rvert \} \int_{\Omega_h}\phi^2+\psi^2.
	\end{equation*}
	Since $a_{12}, a_{22}>0$,
	\begin{equation*}
		a_{12}a_{22}\lvert \phi \psi \rvert (x) \leq \frac{a_{12}a_{22}}{2}(\phi^2+\psi^2)(x)
	\end{equation*}
	for each $x\in \Omega_h$ by Young's inequality, so
	\begin{equation*}
		\int_{\Omega_h}a_{12}a_{22}\phi \psi \geq - \frac{a_{12}a_{22}}{2}(\phi^2+\psi^2).
	\end{equation*}
	Therefore,
	\begin{equation*}
	\begin{aligned}
		I[\phi,\psi] \geq &\min \{ a_{22}d_2,a_{12}d_3\}(\lVert \nabla \phi \rVert _{L^2}^2+\lVert \nabla \psi \rVert _{L^2}^2)\\
		&+\left(\max \{ \lvert a_{11} \rvert a_{22}, a_{12}\lvert a_{21} \rvert\} - \frac{a_{12}a_{22}}{2}\right)(\lVert \phi \rVert _{L^2}^2+\lVert \psi \rVert _{L^2}^2).
	\end{aligned}
	\end{equation*}
	Thus,
	\begin{equation*}
		I[\phi,\psi] \geq (\min \{ a_{22}d_2, a_{12}d_3 \} -\left(\max \{ \lvert a_{11} \rvert a_{22}, a_{12} \lvert a_{21} \rvert \} - \frac{a_{12}a_{22}}{2}\right)C(\lVert \phi \rVert_{L^2}^2+\lVert \psi \rVert _{L^2}^2),
	\end{equation*}
	by Poincar\'e inequality
	\begin{equation*}
		\lVert \phi \rVert _{L^2}^2 + \lVert \psi \rVert _{L^2}^2 \leq C (\lVert \nabla \phi \rVert _{L^2}^2 + \lVert \nabla \psi \rVert _{L^2}^2).
	\end{equation*}
	\begin{equation*}
		I[\phi, \psi ]\geq \alpha \lVert (\phi, \psi ) \rVert _{H_0^1(\Omega_h)^2}^2,
	\end{equation*}
	i.e., the energy is coercive.
\end{proof}

\begin{lemma}
	The energy functional $\mathcal{E}$\cref{energy functional} is lower semi-continuous.
\end{lemma}

\begin{proof}
	Define
	\begin{equation*}
		\mathcal{E}_1[\phi,\psi]=\int_{\Omega_h}a_{22}d_2\lvert \nabla \phi \rvert ^2 + a_{12}d_3\lvert \nabla \psi \rvert ^2.
	\end{equation*}
	Then
	\begin{equation*}
		\mathcal{E}_1[\phi, \psi] = \lVert (\phi, \psi ) \rVert_{H_0^1(\Omega_h)^2}^2
	\end{equation*}
	and is a convex functional. Since a convex functional in the Sobolev space is lower semi-continuous, $\mathcal{E}_1$ is lower semi-continuous. Next, define
	\begin{equation*}
		\mathcal{E}_2[\phi, \psi]=\int_{\Omega_h}a_{11}a_{22}\phi^2+a_{12}a_{21}\psi^2+2a_{12}a_{22}\phi \psi.
	\end{equation*}
	By Rellich-Kondrachov compact embedding,
	\begin{equation*}
		(\phi_n, \psi_n) \rightharpoonup (\phi, \psi) \text{ in }H_0^1(\Omega_h) \Rightarrow(\phi_n,\psi_n)\to (\phi,\psi) \text{ strongly in }L^2.
	\end{equation*}
	Then
	\begin{equation*}
		\mathcal{E}_2[\phi_n,\psi_n] \to \mathcal{E}_2[\phi,\psi],
	\end{equation*}
	i.e., $\mathcal{E}_2$ is lower semi-continuous. Therefore, $\mathcal{E}=\mathcal{E}_1+\mathcal{E}_2$ is also lower semi-continuous.
\end{proof}

\begin{proposition} \label{theorem 4.13}
	Assume the admissible set $\mathcal{A}$ is nonempty. Then there exists $u\in\mathcal{A}$ satisfying
	\begin{equation*}
		\mathcal{E}[u, v]=\min_{(w,x)\in \mathcal{A}} \mathcal{E}[w,x].
	\end{equation*}
\end{proposition}

\begin{proof}
	Set
	\begin{equation*}
		I^*=\min_{(\phi,\psi)\in K}I[\phi,\psi] \swarrow I[\phi_k,\psi_k] \quad \text{for } w_k\in K.
	\end{equation*}
	Since the energy is coercive, $\lVert \phi_k, \psi_k \rVert_{H_0^1(\Omega_h)^2}\leq C$ uniformly and there exists $(\phi_0,\psi_0)\in H^1(\Omega_h)^2$ such that $(\phi_n,\psi_n) \rightharpoonup(\phi_0,\psi_0)$ in $H_0^1(\Omega_h)$.
	Furthermore, 
	\begin{equation*}
		H_1\hookrightarrow L^q \quad \text{for }q<2^*=\frac{2n}{n-2}
	\end{equation*}
	by the $L^p$ theory.
	\begin{equation*}
	\begin{aligned}
		\lvert J[\phi_0,\psi_0] \rvert &= \lvert J[\phi_0,\psi_0]-J[\phi_k,\psi_k]\rvert \\
		&\leq \int_{\Omega_h}\lvert G(\phi_0,\psi_0)-G(\phi_k,\psi_k)\rvert=\int_{\Omega_h}\lvert(\phi_0^2+\psi_0^2)-(\phi_k^2+\psi_k^2)\rvert\\
		&\leq \int_{\Omega_h}\lvert\phi_0-\phi_k\rvert \lvert \phi_0-\phi_k \rvert + \lvert \psi_0 - \psi_k \rvert \lvert \psi_0 + \psi_k \rvert\\
		&\leq \int_{\Omega _h} \lvert \phi_0 - \phi_k \rvert (\lvert \phi_0 \rvert + \lvert \phi_k \rvert) + \lvert \psi_0-\psi_k \rvert (\lvert \psi_0 \rvert + \lvert \psi_k \rvert) \to 0,
	\end{aligned}
	\end{equation*}
	as $k \to \infty$. Therefore, there exists $(\phi, \psi) \in K$ satisfying $\displaystyle I[\phi, \psi]=\min_{[\varphi, \chi] \in K}I[\varphi, \chi]$
\end{proof}

\begin{proposition}
	The minimizer $(w,z)$ of the energy functional \cref{energy functional} of \cref{theorem 4.13} is the solution of the Euler-Lagrange equation
	\begin{equation*}
		\mathcal{L}\begin{pmatrix}
			\phi\\
			\psi
		\end{pmatrix}=2\lambda \begin{pmatrix}
			\phi\\
			\psi
		\end{pmatrix},
	\end{equation*}
	where
	\begin{equation*} \label{euler lagrange}
		\lambda = \begin{pmatrix}
			\lambda_{11}\\
			\lambda_{12}
		\end{pmatrix}.
	\end{equation*}
\end{proposition}
\begin{proof}
	Fix $(u,v)\in H_0^1(\Omega_h)^2$. Then, assume
	\begin{equation*}
		g(\phi, \psi)=\begin{pmatrix} 2\phi \\
		2\psi \end{pmatrix} \neq 0
	\end{equation*}
	almost everywhere within $\Omega_h$.
	Choose any function $(w,z)\in H_0^1(\Omega_h)^2$ such that 
	\begin{equation*}
		\int_{\Omega_h}g(\phi, \psi)(w,z)\neq 0.
	\end{equation*}
	Let
	\begin{equation*}
		j(\tau, \sigma)=\int_{\Omega_h}G((\phi,\psi)+\tau(u,v)+\sigma(w,z)).
	\end{equation*}
	Then, by substituting $\tau=0$ and $\sigma=0$, we can obtain
	\begin{equation*}
		j(0,0)=\int_{\Omega_h}G((\phi,\psi))=0.
	\end{equation*}
	For the case of the derivative with respect to $\sigma$,
	\begin{equation*}
		\frac{\partial j}{\partial \sigma}(\tau, \sigma)=\int_{\Omega_h}g((\phi,\psi)+\tau(u,v)+\sigma(w,z))(w,z),
	\end{equation*}
	and substituting $\tau=0$ and $\sigma=0$,
	\begin{equation*}
		\frac{\partial j}{\partial \sigma}(0,0)=\int_{\Omega_h}g(\phi,\psi)(w,z)\neq 0.
	\end{equation*}
	By Implicit Function Theorem, there exists $C^1$ function $\varphi : \mathbb{R} \to \mathbb{R}$ such that
	\begin{equation*}
		\varphi(0)=0 \text{ and } j(\tau,\varphi(\tau))=0,
	\end{equation*}
	for any sufficiently small $\tau$, i.e., $\lvert \tau \rvert \leq \tau_0$.
	Let
	\begin{equation*}
	\begin{aligned}
		i(\tau) \coloneqq \mathcal{E}[(\phi&, \psi)+\tau(u,v)+\varphi(\tau)(w,z)]\\
		=\int_{\Omega_h}&a_{22}d_2\lvert \nabla (\phi+\tau u+\varphi(\tau)w)\rvert^2+a_{12}d_3\lvert \nabla (\phi+\tau v+\varphi(\tau)z)\rvert^2\\
		&+a_{11}a_{22}(\phi+\tau u+\varphi(\tau)w)^2+2a_{12}a_{22}(\psi+\tau v+\varphi(\tau)w)(\phi+\tau u+\varphi(\tau)z)\\
		&+a_{12}a_{21}(\psi+\tau v+\varphi(\tau)z)^2.
	\end{aligned}
	\end{equation*}
	Then, $i(\tau)$ has the minimum at $\tau=0$, and for
	\begin{equation*}
		\begin{aligned}
			&w(\tau)=(\phi,\psi)+\tau(u,v)+\varphi(\tau)(w,z),\\
			&w'(\tau)=(u,v)+\varphi'(\tau)(w,z),
		\end{aligned}
	\end{equation*}
	\begin{equation*}
	\begin{aligned}
		\frac{\partial i}{\partial w}=&2\int_{\Omega_h}a_{22}d_2\nabla \phi_\tau \cdot \nabla h_1 +a_{12}d_3 \nabla \psi \cdot \nabla h_2 + a_{11}a_{22}\phi h_1\\
		&+a_{12}a_{22}(\phi h_2+\psi h_1)+a_{12}a_{21}\psi h_2,
	\end{aligned}
	\end{equation*}
	and
	\begin{equation*}
	\begin{aligned}
		i'(\tau)=\frac{d}{d\tau}\mathcal{E}[w(\tau)]=2\int_{\Omega_h}&a_{22}d_2\nabla \phi_\tau \cdot u_\tau + a_{12}d_3\nabla \psi_\tau \cdot v_\tau +a_{11}a_{22}\phi_\tau u_\tau \\
		&+a_{12}a_{22}(\phi_\tau v_\tau+\psi_\tau u_\tau)+a_{12}a_{21}\psi _\tau v_\tau,
	\end{aligned}
	\end{equation*}
	where $u_\tau=u+\varphi'(\tau)w$ and $v_\tau=v+\varphi'(\tau)z$.
	Thus,
	\begin{equation*}
	\begin{aligned}
		i'(0)=2\int_{\Omega_h}&a_{22}d_2\nabla \phi \cdot \nabla u+a_{12}d_3 \nabla \psi \cdot \nabla v+a_{11}a_{22}\phi u\\
		&+a_{12}a_{22}(\phi v+\psi u)a_{12}a_{21}\psi v =0,
	\end{aligned}
	\end{equation*}
	for any $(u,v)$, and it is equal to
	\begin{equation*}
		\left \{ \begin{aligned}
			-a_{22}d_2\Delta \phi +a_{11}a_{22}\phi +a_{12}a_{22}\psi =0,\\
			-a_{12}d_3\Delta \psi +a_{12}a_{21}\psi +a_{12}a_{22}\phi =0.
		\end{aligned} \right.
	\end{equation*}
	We can also write
	\begin{equation*}
		\mathcal{L}\begin{pmatrix}
			\phi\\
			\psi
		\end{pmatrix}=0,
	\end{equation*}
	with linear operator $\mathcal{L}$. Let
	\begin{equation*}
	\begin{aligned}
		\lambda \coloneqq &\frac{\int Du\cdot Dw dx}{\int g(u)w dx}\\
		=&\frac{\int a_{22}d_2\lvert \nabla \phi \rvert ^2+a_{12}d_3\lvert \psi \rvert^2 +a_{12}a_{22}\phi^2+2a_{12}a_{22}\phi\psi+a_{12}a_{21}\psi^2}{\int (2\phi,2\psi)\cdot(\phi,\psi)}\\
		=&\frac{\int a_{22}d_2\lvert \nabla \phi \rvert^2+a_{12}d_3\lvert \nabla \psi \rvert^2+a_{12}a_{22}\phi^2+2a_{12}a_{22}\phi\psi+a_{12}a_{21}\psi^2}{2\int \psi^2+\psi^2}.
	\end{aligned}
	\end{equation*}
	Then, the minimizer is a weak solution of
	\begin{equation*}
	\begin{aligned}
		&\int a_{22}d_2 \nabla \phi \cdot \nabla u +a_{12}d_3\nabla \psi \cdot \nabla v +a_{11}a_{22}\phi u+a_{12}a_{22}(\phi v+\psi u)+a_{12}a_{21}\psi v\\
		&=\lambda \int (2\phi,2\psi) \cdot (u,v),
	\end{aligned}
	\end{equation*}
	for any $(u,v)\in H_0^1(\Omega_h)^2$.
	By integration by parts and boundary conditions,
	\begin{equation*}
		\left \{ \begin{aligned}
			-a_{22}d_2\Delta \phi +a_{11}a_{22}\phi +a_{12}a_{22}\psi =2\lambda \phi,\\
			-a_{12}d_3\Delta \psi +a_{12}a_{21}\psi +a_{12}a_{22}\phi =2\lambda \psi.
		\end{aligned} \right.
	\end{equation*}
	It means that the minimizer $(\phi,\psi)$ satisfies
	\begin{equation*}
		\mathcal{L}\begin{pmatrix}
			\phi\\
			\psi
		\end{pmatrix}=2\lambda \begin{pmatrix}
			\phi\\
			\psi
		\end{pmatrix}.
	\end{equation*}
\end{proof}

\begin{lemma}
	The solution of the Euler-Lagrange equation \cref{euler lagrange} is the super solution of the equations \cref{DFE ODE}. Furthermore, the solution of the Euler-Lagrange equation \cref{euler lagrange} is the super solution of the equations \cref{modified free boundary SEIS4} with \cref{modified free boundary SEIS4 cond}.
\end{lemma}
\begin{proof}
	If we substitue $(\phi, \psi)$ to \cref{euler lagrange},
	\begin{equation*}
		\left \{ \begin{aligned}
			&\mathcal{L}\begin{pmatrix}
				\phi\\
				\psi
			\end{pmatrix}=2\lambda\begin{pmatrix}
				\phi\\
				\psi
			\end{pmatrix}\geq 0, \qquad &&\text{in } \Omega_\infty,\\
			&\begin{pmatrix}
				\phi\\
				\psi
			\end{pmatrix}\geq 0, &&\text{on } \partial \Omega_\infty.
		\end{aligned} \right.
	\end{equation*}
	Thus, $(\phi, \psi)$ is the super solution of the equations \cref{DFE ODE}. Since the equations \cref{modified free boundary SEIS4} with \cref{modified free boundary SEIS4 cond} are continuous with respect to $t$ and the equations \cref{DFE ODE} with \cref{modified free boundary SEIS4 cond} are the limit of the equations \cref{modified free boundary SEIS4} with \cref{modified free boundary SEIS4 cond}, $(\phi, \psi)$ is also the super solution of the equations \cref{modified free boundary SEIS4} with \cref{modified free boundary SEIS4 cond}.
\end{proof}

\begin{lemma}
	If $\mathcal{R}_0^{\max}(=\frac{\alpha A}{\min_{i=1,2,3}\{\mu_i\}}\min\{r_1+\mu_2,r_2+\mu_3\})<1$, the linear operator $\mathcal{L}$ is negative semi--definite.
\end{lemma}
\begin{proof}
	It is obvious that
	\begin{equation*}
		\beta_1+r_1+\mu_2\geq 0,
	\end{equation*}
	and 
	\begin{equation*}
		p \alpha \frac{A}{\mu_1}-(r_2+\mu_3)\leq 0.
	\end{equation*}
	If $\mathcal{R}_0^{\max}<1$,
	\begin{equation*}
		\frac{\alpha A}{\mu_1}<r_1+\mu_2, \text{ and } \frac{\alpha A}{\mu_1}<r_2+\mu_3.
	\end{equation*}
	Since $p<1$, 
	\begin{equation*}
		\frac{\alpha A}{p \mu_1}<\frac{r_1+\mu_2}{p}<r_1+\mu_2,
	\end{equation*}
	so
	\begin{equation*}
		\frac{\alpha A}{\mu_1}<(r_1+\mu_2)p.
	\end{equation*}
	To show that $\mathcal{L}$ is negative definite, we have to show that
	\begin{equation*}
		-(\beta_1+r_1+\mu_2)(p\alpha \frac{A}{\mu_1}-(r_2+\mu_3))-\beta_1(1-p)\alpha\frac{A}{\mu_1}\geq 0.
	\end{equation*}
	Here,
	\begin{equation*}
	\begin{aligned}
		&-(\beta_1+r_1+\mu_2)(p\alpha \frac{A}{\mu_1}-(r_2+\mu_3))-\beta_1(1-p)\alpha\frac{A}{\mu_1}\\
		&=\beta_1(r_2+\mu_3)-(r_1+\mu_2)p\frac{\alpha A}{\mu_1}+(r_1+\mu_2)(r_2+\mu_3)-\beta_1(1-p)\frac{\alpha A}{\mu_1}\\
		&=(\beta_1+r_1+\mu_2)(r_2+\mu_3)-((r_1+\mu_2)p+\beta_1)\frac{\alpha A}{\mu_1}\\
		&\geq (\beta_1+r_1+\mu_2)(r_2+\mu_3-\frac{\alpha A}{\mu_1})\geq \beta_1+r_1+\mu_2\geq 0.
	\end{aligned}
	\end{equation*}
\end{proof}

The below theorem shows the convergent speed of the solution of \cref{DFE ODE}.
\begin{theorem}
	If $\mathcal{R}_0<1$, then for every $(u_0,v_0)\in L^2(\Omega_\infty)^2$, we have 
	\begin{equation*}
		\left| e^{\lambda_1t} \begin{pmatrix}
			u(s,t)\\
			v(s,t)
		\end{pmatrix} -a_1 \begin{pmatrix}
			\phi(s)\\
			\psi(s)
		\end{pmatrix} \right|  \leq C \cdot e^{-(\lambda_2-\lambda_1) t}
	\end{equation*}
	where $\lambda_1$ and $\lambda_2$ are eigenvalues obtained in the previous therorems, $a_1$ is the coefficient from eigenfunction expansion for initial, and $(\phi(s),
		\psi(s))$
	is the corresponding eigenfunction associated with $\lambda_1$.

\end{theorem}
\begin{proof}
	Using the previous results on the eigenvalue problems, put
	\begin{equation*}
		e^{\lambda_1t}\begin{pmatrix}
u(s,t)\\
v(s,t)	
\end{pmatrix}
=a_1\begin{pmatrix}
			\phi(s)\\
			\psi(s)
		\end{pmatrix}+e^{-(\lambda_2-\lambda_1)t}\begin{pmatrix}
			\varphi(s) \\
			\chi(s)
		\end{pmatrix},
	\end{equation*}
	where $\lambda_1>\lambda_2$ are the biggest eigenvalues and $\begin{pmatrix}
			\phi(s)\\
			\psi(s)
		\end{pmatrix}$ and $\begin{pmatrix}
			\varphi(s) \\
			\chi(s)
		\end{pmatrix}$ are corresponding eigenfunctions respectively.
	Then, the solution of \cref{modified free boundary SEIS4} converges at the rate $O(e^{-(\lambda_2-\lambda_1)})$.
\end{proof}

\nocite{Lin_2007}
\nocite{Kim_2013}
\nocite{Kim_2008}
\nocite{MENG_2013}
\nocite{Yihong_2010}
\nocite{Alaa_2013}
\nocite{Chen_2000}
\nocite{Arino_2008}
\nocite{Mimura_1985}
\nocite{Mimura_1987}
\nocite{Mimura_1986}
\nocite{Friedman_1964}
\nocite{Friedman_1982}

%==========================================================================
%
% 							References
%
%==========================================================================

\end{document}